\documentclass[12pt]{amsart}

\usepackage{amsthm,amsmath,amssymb}
\RequirePackage[colorlinks,citecolor=blue,urlcolor=blue]{hyperref}
\RequirePackage[OT1]{fontenc}
\RequirePackage[numbers]{natbib}
\usepackage{enumerate}
\usepackage{cleveref}
\usepackage{graphicx}
\graphicspath{{./figs/}}
\usepackage{caption}
\usepackage{subcaption}
\usepackage{bbm}

\usepackage{xcolor}

\numberwithin{equation}{section}
\theoremstyle{plain}
\newtheorem{theorem}{Theorem}[section]
\newtheorem{defn}[theorem]{Definition}
\newtheorem{assumption}[theorem]{Assumption}
\newtheorem{lemma}[theorem]{Lemma}
\newtheorem{corollary}[theorem]{Corollary}
\theoremstyle{remark}
\newtheorem{remark}[theorem]{Remark}

\crefname{assumption}{assumption}{assumptions}
\crefname{theorem}{Theorem}{Theorems}
\crefformat{equation}{(#2#1#3)} 
\crefname{remark}{remark}{remarks}

\newcommand{\pushright}[1]{\ifmeasuring@#1\else\omit\hfill$\displaystyle#1$\fi\ignorespaces}
\newcommand{\pushleft}[1]{\ifmeasuring@#1\else\omit$\displaystyle#1$\hfill\fi\ignorespaces}

\newcommand{\negA}[1]{(-A)^{#1}}

\newcommand{\DAs}[1]{\mathcal{D}((-A)^{#1})}
\newcommand{\norm}[1]{\left\lVert#1\right\rVert}
\newcommand{\Bnorm}[1]{\left\lVert#1\right\rVert_{L_{0}^{2}}}

\newcommand{\Onorm}[2]{\left\lVert#1\right\rVert_{\mathcal{L}(#2)}}
\newcommand{\HSnorm}[3]{\left\lVert#1\right\rVert_{HS(#2,#3)}}
\newcommand{\Econd}[1]{\mathbb{E}\left[ #1 \hspace{0.1cm} \middle| \mathcal{G}_{n} \right]}
\newcommand{\Econdclose}[1]{\left. #1\hspace{0.1cm} \middle| \mathcal{G}_{n} \right]}
\newcommand{\Econdopen}[1]{\mathbb{E}\left[ #1 \right.}
\newcommand{\EcondF}[1]{\mathbb{E}\left[ #1 \hspace{0.1cm} \middle| \mathcal{F}_{\rho} \right]}
\newcommand{\cond}[2]{\hspace{0.1cm} | \mathcal{#1}_{#2} ]}

\newcommand{\Lcondnorm}[1]{\norm{#1}_{L^{2}(\Omega|\mathcal{G}_{n};H)} }
\newcommand{\Lnorm}[1]{\norm{#1}_{L^{2}(\Omega;H)} }

\newcommand{\SE}[1]{ S_{h}(T-t_{#1})E_{k}(t_{#1}) }
\newcommand{\SM}[1]{ S_{h}(T-t_{#1}) }
\newcommand{\ito}{It\^o}
\newcommand{\itos}{It\^o's}

\newcommand{\adptone}{ASETD1}
\newcommand{\adptzero}{ASETD0}
\newcommand{\stp}{NSEE}
\newcommand{\tamed}{TEM}
\newcommand{\tamedexp}{TSETD0}
\newcommand{\breakgap}{\qquad}
\newcommand{\bigbreakgap}{\qquad\qquad}
\newcommand{\LF}{L_{F}}
\newcommand{\LB}{L_{B1}}
\newcommand{\LBr}{L_{B2}}
\newcommand{\Kreg}{K_{1}}
\newcommand{\Cs}{C_{s}}
\newcommand{\KF}{K_{2}}
\newcommand{\KFsq}{K_{3}}
\newcommand{\KB}{K_{4}}

\newcommand{\CT}{C_{T}}
\newcommand{\CI}[1]{C_{#1}}
\newcommand{\Cmain}{C(T,\rho,X_{0})}
\newcommand{\Cspatial}{C(T,X_{0})}
\newcommand{\supT}{\sup_{s\in [0,T]}}

\newcommand{\supTE}[1]{\supT \mathbb{E} \left[ #1 \right] }
\newcommand{\Konep}[1]{\tilde{K}_{1,#1}}

\begin{document}
\bibliographystyle{apa}

\title[Adaptive time-stepping for SPDEs]{Adaptive time-stepping for Stochastic Partial Differential Equations with non-Lipschitz drift}

\author{Stuart Campbell}
\address{Department of Mathematics, Maxwell Institute, MACS, Heriot-Watt University, Edinburgh, EH14 4AS, UK}
\email{sc58@hw.ac.uk}

\author{Gabriel Lord}
\address{Department of Mathematics, Maxwell Institute, MACS, Heriot-Watt University, Edinburgh, EH14 4AS, UK}
\address{Department of Mathematics, IMAPP, University of Radboud, 6500 GL Nijmegen, The Netherlands}
\email{g.j.lord@hw.ac.uk}

\keywords{Stochastic Partial Differential Equations, strong convergence, Allen-Cahn equation, adaptive, non-globally Lipschitz}
\subjclass[2010]{65C30, 60H15, 65N12}

\begin{abstract}
We introduce an explicit, adaptive time-stepping scheme for the simulation of SPDEs with one-sided Lipschitz drift coefficients. Strong convergence rates are proven for the full space-time discretisation with multiplicative trace-class noise by considering the space and time discretisation separately. Adapting the time-step size to ensure strong convergence is shown numerically to produce more accurate solutions when compared to alternative fixed time-stepping strategies for the same computational effort.
\end{abstract}

\maketitle

\section{Introduction}

We construct an adaptive time-stepping method to ensure strong convergence of the numerical approximation to a semilinear stochastic partial differential equation (SPDE) with one-sided Lipschitz drift and multiplicative trace-class noise. We consider SPDEs of the form 
\begin{equation}\label{eq_SPDE_def}
dX = \left[ -AX + F(X) \right] dt + B(X) dW, \quad X(0) = X_{0}, t\in [0,T], T>0,
\end{equation}
on a separable Hilbert space $H=L^{2}(\mathbb{D})$, with domain $\mathbb{D} \subset \mathbb{R}^{d}$. The linear operator $-A$ is assumed to be the generator of an analytic semigroup, $S(t):=e^{-tA}$. Under additional assumptions, given in \Cref{sec:assumptions}, existence and uniqueness of the mild solution,
\begin{equation}\label{eq:mild_sol}
X(t) = S(t)X_{0} + \int_{0}^{t} S(t-s)F(X(s)) ds
+
\int_{0}^{t} S(t-s) B(X(s)) dW(s), 
\end{equation}
can be proved. It is known \citep{Hutzenthaler2011} that the standard Euler-Maruyama method for stochastic differential equations (SDEs) is not guaranteed to converge strongly when the drift does not satisfy a global Lipschitz condition. Unfortunately many applications of interest for both SDEs and SPDEs have non-Lipschitz drift functions. There are several directions that have been taken to ensure strong convergence in the case of non-globally Lipschitz drift for both SDEs and SPDEs. Firstly, implicit methods are strongly convergent for SDEs with non-Lipschitz drift, however are more computationally expensive than explicit methods. This can render them impractical, particularly for large-scale Monte-Carlo estimations or inverse problems. For this reason there has been much interest in the construction of explicit methods that ensure strong convergence under non-Lipschitz drift. 

Since their introduction for one-sided Lipschitz SDEs in \citep{hutzenthaler_tamed2012}, the class of tamed methods has seen widespread interest. These methods work by perturbing the drift function to ensure bounded moments of the numerical solution. Roughly speaking, the perturbation of the drift ensures bounded growth of the new drift function, with respect to the time-step size, and the perturbed drift function converges to the original drift function as the time-step is decreased towards zero.

An alternative method to ensure strong convergence for SDEs was derived in \citep{Fang2016,kellylord2017adaptive}. There, the time-step size is adapted to ensure convergence whilst the drift function is unchanged. It was shown that such an adaptive method can result in more accurate simulations when compared to tamed, fixed time-step methods and therefore lead to more efficient multi-level Monte-Carlo simulations. 

Compared to the SDE literature, there are far fewer methods available that ensure strong convergence for general one-sided Lipschitz SPDEs. In fact there are only two that we are aware of. Firstly, convergence and stability of a tamed Euler-Maruyuama scheme for SPDEs was proved in the variational setting in \citep{Gyongy2016}, which we later refer to as \tamed{} (Tamed Euler-Maruyuama) in \Cref{sec:numerics}. Secondly an  alternative method of taming the drift was introduced for SPDEs in \citep{Jentzen2015}, where the contributions of the nonlinear drift and diffusion are ``switched off'' if the nonlinearities are too large when compared to the scheme's time-step. The switching behaviour ensures control of the numerical scheme moments and allows for strong convergence to be proven. We later refer to this scheme as \stp{} (Nonlinearities Stopped Exponential Euler) in \Cref{sec:numerics}.

Strong convergence of a spatial semi-discrete approximation is achieved for general one-sided Lipschitz drift and additive noise via random PDE equivalence in \citep{cui2019}. Convergence in probability was proven in \citep{Anton2018} for the one dimensional non-linear heat equation. We make the distinction in this paper that we do not restrict to any specific drift function and consider multiplicative noise for the full spatial-temporal discretisation. 

For particular SPDEs with non-Lipschitz drift, there are many more schemes that ensure convergence, for example in \citep{Brehier2018b} and \citep{Wang2018}, strong convergence results are proven for the Allen-Cahn equation with additive space-time white noise via splitting.

In this paper we propose an adaptive time-stepping method for SPDEs that ensures strong convergence in the presence of non-Lipschitz drift with multiplicative noise. This is achieved by reducing the time-step of the scheme when the numerical drift response is large, ensuring at most linear growth of the numerical drift response with respect to the solution of the numerical scheme. The underlying numerical method is based on a finite element or spectral Galerkin exponential integrator, which was considered for multiplicative noise with a fixed time-step in e.g. \citep{lord2013stochastic}, \citep{Tambue2016}.

The content of the paper is organised as follows, in \Cref{sec:assumptions} we formalise the setting and assumptions made, as well as stating some well known properties of the semigroup $S(t)$. \Cref{sec:galerkin} is a brief overview of the necessary Galerkin methods for the spatial discretisation. In \Cref{sec:fullydiscrete} we give the fully discrete approximation and define our numerical scheme. The idea of what we call an \emph{admissible time-stepping strategy} is defined in \Cref{sec:time_stepping}. This ensures there is no unwanted growth in the numerical solution and is key to proving our main results of strong convergence rates, stated in \Cref{sec:mainresults}. Some numerical comparisons to fixed time-step methods are shown in \Cref{sec:numerics}. Due to the nature of the time-stepping strategies introduced, we require a conditional \ito{} isometry and conditional regularity in time of the mild solution, which we prove in \Cref{sec:cond_ito}. We conclude with the proofs of the main results in \Cref{sec:proofmain}.
 
\section{Setting and Assumptions}\label{sec:assumptions}
We consider the SPDE in \Cref{eq_SPDE_def} on a Hilbert space $H$ with norm $\norm{\cdot}$ and inner product $\langle \cdot,\cdot \rangle$. The linear operator $-A$ is assumed to be the generator of an analytic semigroup, $S(t):=e^{-tA}$ for $t \geq 0$. Both $F$ and $B$ are possibly nonlinear in $X$, with $F$ satisfying a monotone condition and a polynomial bound. The diffusion coefficient $B$ is assumed to be globally Lipschitz. By $\mathcal{L}(H)$ we denote the set of bounded linear operators from $H\rightarrow H$.

For the noise, we define a $Q$-Wiener process on the probability space $(\Omega,\mathcal{F},\mathbb{P},\{F_{t}\}_{t\geq 0})$ taking values in a Hilbert space $U$, with covariance operator $Q\in\mathcal{L}(U)$. 
We write $W(t)$ as
\begin{equation}
W(t) = \sum_{j=1}^{\infty} \sqrt{q_{j}}\chi_{j}\beta_{j}(t), \nonumber
\end{equation}
where $\{\chi_{j}\}_{j\in \mathbb{N}}$ are the eigenfunctions of $Q$
that form an orthonormal basis of $U$
(see \citep{DaPrato1992,lord2014introduction}). The $q_{j}$ are real
eigenvalues of $Q$ such that $q_{j}\geq 0$ and
$\sum_{j=1}^{\infty}q_{j}<\infty$. Finally the $\beta_{j}$ are
i.i.d. $\mathcal{F}_{t}$-adapted one-dimensional Brownian motions. 
To complete the setting for the diffusion coefficient $B$, we define a subset of the Hilbert space $U$ as $U_{0}:=\{ Q^{1/2} u: u\in U \}$. We take $B\in L_{0}^{2}$, the set of Hilbert-Schmidt operators from $U_{0}$ to $H$. That is $B: U_{0} \rightarrow H$ satisfies
\begin{equation}
\Bnorm{B} := 
\left( \sum_{j=1}^{\infty} \norm{ BQ^{1/2} \chi_{j}}^{2} \right)^{1/2} < \infty \nonumber,
\end{equation}
with $\{\chi_{j}\}$ an orthonormal basis of $U$.

We are interested in strong convergence of the numerical scheme to the true solution, that is convergence with respect to the norm
\begin{equation}
\Lnorm{\cdot} := \left( \mathbb{E}[\norm{\cdot}^{2} ]\right)^{\frac{1}{2} }. \nonumber
\end{equation}
During the analysis we need to consider conditional expectations in time to ensure the adaptive scheme maintains normally distributed Wiener increments and does not introduce any bias. Therefore we introduce the following notation
\begin{equation}
\Lcondnorm{\cdot} := 
\left( \Econd{\norm{\cdot}^{2}} \right)^{\frac{1}{2} }, \nonumber
\end{equation}
which is required when discussing the full space-time convergence of the adaptive scheme. The discrete-time filtration $\mathcal{G}_{n}$ will be defined later in \Cref{defn:Gn}. We now recall some standard properties of the semigroup $S(t)$, and state the precise assumptions on $F$ and $B$.
\begin{assumption}\label{assump_A_semigroup_generator}
The operator $-A:\mathcal{D}(-A) \rightarrow H$ is the generator of an analytic semigroup $S(t) = e^{-tA},t\geq 0$.
\end{assumption}
Let $\lambda_{j}>0$ and $\phi_{j}$ be the eigenvalues and eigenfunctions of $-A$ with $\lambda_{j}\leq \lambda_{j+1}$. Since $\lambda_{j}>0$, the exponential $e^{-tA}u
:=\sum_{j=1}^{\infty} e^{-\lambda_{j}t}\langle u, \phi_{j} \rangle \phi_{j}
\leq \sum_{j=1}^{\infty} \langle u, \phi_{j} \rangle \phi_{j}=:u$ for $t\geq 0$. Hence $\Onorm{S(t)}{H}=1$.
\begin{lemma}\label{lemma_Stv-v}
Under \Cref{assump_A_semigroup_generator} the following bounds hold for $0\leq \alpha \leq 1$, $\beta \geq 0$ and some $\Cs>0$,
\begin{align}
\norm{\negA{\beta}S(t)}_{\mathcal{L}(H)} \leq \Cs t^{-\beta}  \quad\text{ for } t>0 \nonumber \\
\norm{\negA{-\alpha}(I-S(t))}_{\mathcal{L}(H)} \leq \Cs t^{\alpha}  \quad\text{ for } t\geq 0 .\nonumber
\end{align}
\end{lemma}
\begin{proof}
See \citep{henry2006geometric}.
\end{proof}
For $v \in \DAs{\alpha/2}$, $\alpha \in \mathbb{R}$, we use throughout the notation $\norm{v}_{\alpha}:= \norm{(-A)^{\alpha/2}v}$. Under this norm, and corresponding inner product $\langle (-A)^{\alpha/2}v, (-A)^{\alpha/2}v \rangle$, the space $\DAs{\alpha/2}$ is a Hilbert space. We apply \Cref{lemma_Stv-v} to attain the following useful inequality 
\begin{equation}\label{eq_I-St_v_bnd}
\norm{(I-S(t))v}^{2} 
	= 
	\norm{\negA{-\alpha/2}(I-S(t))\negA{\alpha/2}v}^{2} 
	\leq 
	\Cs t^{\alpha}\norm{v}_{\alpha}^{2},
\end{equation}
for $v \in \DAs{\alpha/2}$ and $0\leq \alpha \leq 1$.

The usual global Lipschitz assumption on $F$ is weakened to the following one-sided Lipschitz condition and a polynomial bound on the Fr\'echet derivative. 
\begin{assumption}\label{assump_F_local_lipschitz}
The function $F$ satisfies a one sided Lipschitz growth condition. That is, there exists a constant $\LF{}>0$ such that for all $X,Y\in H$,
\begin{equation}
\langle F(X)-F(Y),X-Y \rangle \leq \LF{} \norm{X-Y}^{2}. \nonumber
\end{equation}
\end{assumption}
\begin{assumption}\label{assump_DF_growth}
The Frechet derivative of $F$ is bounded polynomially, that is for some $c_{1},c_{2}\in(0,\infty)$,
\begin{equation}
\Onorm{ DF(X) }{H} \leq c_{1}(1 + \norm{X}^{c_{2}}). \nonumber
\end{equation}

\end{assumption}
Under \Cref{assump_DF_growth}, the growth of $F$ can be bound polynomially by 
\begin{equation}\label{eq_F_growth}
\norm{F(X)} \leq c_{3}(1 + \norm{X}^{c_{2}+1}),
\end{equation}
with $c_{3}:= 2c_{1} + \norm{F(0)}$.
A global Lipschitz assumption on the diffusion term $B$ is assumed throughout this work. 
\begin{assumption}\label{assump_B_lip}
The diffusion coefficient $B$ satisfies a global Lipschitz growth condition, i.e. there exists a constant $\LB>0$ such that
\begin{equation}
\norm{ B(X)-B(Y) }_{L_{0}^{2}} \leq \LB \norm{X-Y},\quad \forall X,Y \in H. \nonumber
\end{equation}
\end{assumption}
\begin{assumption}\label{assump_B_growth}
For some fixed $r \in [0,1)$  and fixed constant $\LBr>0$ the following bound holds
\begin{equation}
\norm{ \negA{r/2}B(X) }_{L_{0}^{2}} \leq \LBr (1+\norm{X}_{r}),\quad \text{for all } X \in \DAs{r/2}. \nonumber
\end{equation}
\end{assumption}
The parameter $r$ in \Cref{assump_B_growth} determines the spatial regularity of the noise and hence the mild solution. The classical result of \citep{DaPrato1992} proves existence and uniqueness for the mild solution of \Cref{eq_SPDE_def} under global Lipschitz assumption on the drift function $F$. We require a similar result for the weaker assumptions on $F$, which was shown in \citep{Jentzen2015}.
\begin{theorem}\label{thm:moment_bounds}
Under \Cref{assump_A_semigroup_generator,assump_F_local_lipschitz,assump_DF_growth,assump_B_lip,assump_B_growth} with initial data $X_{0} \in L^{2}(\mathbb{D},\DAs{1/2})$, there exists a mild solution, unique up to equivalence given by
\begin{equation}\label{eq_SPDE_mild_sol}
X(t) = S(t)X_{0} + \int_{0}^{t} S(t-s)F(X(s)) ds
+
\int_{0}^{t} S(t-s) B(X(s)) dW(s), \nonumber
\end{equation}
Further that for all $p \geq 2$ and all $\eta \in [0,1)$,
\begin{equation}\label{eq_moment_bounds}
\sup_{t\in [0,T]} \mathbb{E}\norm{ X(t) }^{p}_{\eta} < \infty. 
\end{equation}
\end{theorem}

\begin{proof}
See \citep{Jentzen2015}, in particular Proposition 7.3 for the moment bounds. 
\end{proof}

\section{Spatial Discretisation}\label{sec:galerkin}
We now turn our attention to the numerical scheme and present a brief overview of the relevant Galerkin methods required for the spatial discretisation of the SPDE problem in this section. A detailed presentation of these methods can be found in e.g. \citep{thomee2006}. The parameter $h\in (0,1)$ will determine the spatial granularity of the approximation. Let $V_{h} \subset V:=\DAs{\frac{1}{2}}$ be a sequence of finite dimensional subspaces of $V$ and define the Riesz projection $R_{h}: V \rightarrow V_{h}$ by
\begin{equation}
(-AR_{h}v,w_{h}) = a(v,w_{h}) := (A^{1/2}v,A^{1/2}w_{h}) \quad \text{for all } v\in V, w_{h}\in V_{h}. \nonumber
\end{equation}
We require an assumption on the order of the approximation of elements in $V$ by elements in $V_{h}$. We show, in \cref{sec:FEM,sec:spectral}, this assumption is satisfied by the usual finite element and spectral Galerkin methods.
\begin{assumption}
For $V_{h}\subset V$ there exists a constant $C$ such that 
\begin{equation}
\norm{R_{h}v-v} \leq Ch^{s}\norm{v}_{s} \quad \text{for all } v\in \DAs{s/2}, \quad s\in\{1,2\}, h\in (0,1). \nonumber
\end{equation}
\end{assumption}
As usual we first consider the linear deterministic problem. That is, given $u(0)=v$ find $u\in V$ such that
\begin{equation}
\frac{du}{dt} = -Au, \text{ for } t\in (0,T]. \nonumber
\end{equation}
We define $-A_{h}:V_{h}\rightarrow V_{h}$ as the discrete version of the linear operator $A$ via the relation
\begin{equation}
(-A_{h}v_{h},w_{h}) = a(v_{h},w_{h}) \text{ for all } v_{h},w_{h} \in V_{h}. \nonumber 
\end{equation}
The discrete operator $-A_{h}$ is also the generator of an analytic semigroup $S_{h}(t):= e^{-tA_{h}}$ and therefore satisfies similar smoothing properties as $S(t)$ does in \Cref{lemma_Stv-v} \citep{thomee2006}. 

We require two additional projection operators to discretise the SPDE, firstly the orthogonal projection operator $P_{h}:L^{2}(\mathcal{D})\rightarrow V_{h}$ defined by
\begin{equation}
(P_{h}v,w) = (v,w) \quad \text{ for all } v \in L^{2}(\mathcal{D}), w\in V_{h}. \nonumber
\end{equation}
Note that $\norm{P_{h}v}\leq \norm{v}$ for all $v \in L^{2}(\mathcal{D})$ and therefore the operator norm $\Onorm{P_{h}}{H} \leq 1$.

The second projection operator $P_{J}$, is the projection of $u \in L^{2}(\mathbb{D})$ onto the first $J$ terms of an orthonormal basis $\{ \chi_{j} \}$ of $L^{2}(\mathbb{D})$, 
\begin{equation}\label{eq:PJ}
P_{J} u := \sum_{j=1}^{J} \langle \chi_{j}, u \rangle \chi_{j}. 
\end{equation}
It is easily seen that $\Onorm{P_{J}}{H}\leq 1$. We finish this section with two examples that fall under the above framework, firstly the standard finite element method and secondly the spectral Galerkin method.

\subsection{Finite Element Method}\label{sec:FEM}
Assume that $\mathbb{D} \subset \mathbb{R}^{d}$ for $d=1,2,3$ and $\mathbb{D}$ is a bounded convex domain that either has a smooth boundary or is a convex polygon. We define two spaces, $\mathbb{H}\subset V$ which depend on the boundary conditions of the SPDE. For Dirichlet boundary conditions we have that
\begin{equation}
V = \mathbb{H} = H_{0}^{1}(\mathbb{D}) \nonumber
\end{equation}
and for Robin boundary conditions
\begin{align*}
V &= H^{1}(\mathbb{D}) \\
\mathbb{H} &= \{ v \in H^{2}(\mathbb{D}) \frac{ \partial v}{\partial n} + c_{r}v = 0 \text{ on } \partial \mathbb{D} \}, \quad c_{r}\in \mathbb{R}.
\end{align*}
The space $\mathbb{H}$ encodes the boundary conditions of the SPDE. The following characterisations of $\DAs{s/2}$ for $s\in\{1,2\}$ are well known, e.g. \citep{Larsson1992,elliot1992},
\begin{align*}
\norm{v}_{H^{s}(\mathbb{D})} 
&\equiv \norm{v}_{s}, \text{ for all } v\in \DAs{s/2} \\
\DAs{s/2} 
&= \mathbb{H} \cap H^{s}(\mathbb{D}) \quad (\text{Dirichlet B.C.s}) \\
\DAs{} &= \mathbb{H}, \quad \DAs{1/2} = H^{1}(\mathbb{H}) \quad (\text{Robin B.C.s}).
\end{align*}
Let $\mathcal{T}_{h}$ be a $d$-dimensional triangulation of the domain $\mathbb{D}$ satisfying the usual regularity assumptions with maximal mesh size $h$. Then $V_{h} \subset V$ is the space of continuous functions that are piecewise linear over the mesh $\mathcal{T}_{h}$.

\subsection{Spectral Galerkin Method}\label{sec:spectral}
For the spectral Galerkin method to be applicable, we require that the combination of domain $\mathbb{D} \subset \mathbb{R}^{d}$, operator $-A$, and boundary data, has eigenvalues and orthonormal eigenfunctions that are explicitly computable. For example when $\mathbb{D}=[0,a]$ with Dirichlet or periodic boundary conditions and $-A$ is equal to the Laplacian. In this setting we can compute the eigenvalues $\lambda_{j}$ and eigenfunctions $\phi_{j}$ of $A$. By setting $h:= \lambda_{J+1}^{-1/2}$ for some $J\in \mathbb{N}$, the spaces $V_{h}$ are then defined as
\begin{equation}
V_{h}:= \text{span} \{ \phi_{j}: j=1,\hdots J\} \nonumber.
\end{equation}

\section{Fully Discrete Scheme}\label{sec:fullydiscrete}
In the notation of the previous section, the semi-discrete problem can be defined as the mild solution to
\begin{equation}\label{eq:semidiscrete_def}
dX_{h} = \left[ -A_{h}X_{h} + P_{h}F(X_{h}) \right] dt + P_{h}B(X_{h}) P_{J}dW, \quad
X_{h}(0) = P_{h}X_{0},
\end{equation}
that is
\begin{equation}\label{eq:SPDE_SD_sol}
\begin{split}
X_{h}(T) &= S_{h}(T)P_{h}X_{0} + 
\int_{0}^{T} S_{h}(T-s)P_{h}F(X_{h}(s)) ds \\
&\breakgap +
\int_{0}^{T} S_{h}(T-s)P_{h}B(X_{h}(s)) P_{J}dW(s). 
\end{split}
\end{equation}
To obtain a fully discrete scheme, we apply the approximations in time
\begin{alignat}{3}\label{eq:phi1_approx}
F(X_{h}(s)) 
	&\approx 
	F(X_{h}(t_{n})),
	& s\in [t_{n}, t_{n+1}], \nonumber 
	\\
S_{h}(t_{n+1}-s)P_{h}B(X_{h}(s)) 
	&\approx 
	S_{h}(\Delta t_{n+1})P_{h}B(X_{h}(t_{n})), \quad
	& s\in [t_{n}, t_{n+1}], 
\end{alignat}
to \cref{eq:SPDE_SD_sol}, which yields
\begin{equation}\label{eq_SPDE_num_sol}
\begin{split}
X_{h}^{N} &= S_{h}(T)P_{h}X_{0} + 
\sum_{k=0}^{N-1}\bigg{(}\int_{t_{k}}^{t_{k+1}} S_{h}(T-s)P_{h}F(X_{h}^{k})ds \bigg{)} \\
&\breakgap +
\sum_{k=0}^{N-1}\int_{t_{k}}^{t_{k+1}}S_{h}(T-t_{k})P_{h}B(X_{h}^{k})P_{J} dW(s). 
\end{split}
\end{equation}
Re-writing over one step leads to our adaptive numerical scheme \adptone{},
\begin{equation}\label{eq:scheme_adapt1}
\begin{split}
X_{h}^{n+1} &= e^{-\Delta t_{n+1}A_{h}}\left( X_{h}^{n}  + P_{h}B(X_{h}^{n})P_{J} \Delta W_{n+1} \right) \\
&\breakgap + 
\Delta t_{n+1} \varphi_{1}(\Delta t_{n+1}A_{h}) P_{h}F(X_{h}^{n}), 
\end{split}
\end{equation}
with the standard $\varphi_{1}$ function defined by $\varphi_{1}(z):= \frac{e^{z}-1}{z}.$
We can rewrite \Cref{eq:scheme_adapt1} as
\begin{align}
X_{h}^{n+1} 
&= X_{h}^{n}  + P_{h}B(X_{h}^{n})P_{J} \Delta W_{n+1} \nonumber \\
&+ 
\Delta t_{n+1} \varphi_{1}(\Delta t_{n+1}A_{h})\left( 
A_{h}(X_{h}^{n}  + P_{h}B(X_{h}^{n})P_{J} \Delta W_{n+1}) + 
P_{h}F(X_{h}^{n}) \right),  \nonumber
\end{align}
in order to compute only a single matrix exponential at each time-step, which may be more efficient for finite element spatial discretisations. 
If instead we apply the approximations in time
\begin{alignat}{3}\label{eq:phi0_approx}
S_{h}(t_{n+1}-s)P_{h}F(X_{h}(s)) 
	&\approx 
	S_{h}(\Delta t_{n+1})P_{h}F(X_{h}(t_{n})),
	& s\in [t_{n}, t_{n+1}], \nonumber 
	\\
S_{h}(t_{n+1}-s)P_{h}B(X_{h}(s)) 
	&\approx 
	S_{h}(\Delta t_{n+1})P_{h}B(X_{h}(t_{n})), \quad
	& s\in [t_{n}, t_{n+1}],
\end{alignat}
into \cref{eq:SPDE_SD_sol} we obtain the solution
\begin{equation}
\begin{split}
X_{h}^{N} &= S_{h}(T)P_{h}X_{0}  + 
\sum_{k=0}^{N-1}\bigg{(}\int_{t_{k}}^{t_{k+1}} S_{h}(T-t_{k})P_{h}F(X_{h}^{k})ds \bigg{)} \\
&\breakgap +
\sum_{k=0}^{N-1}\int_{t_{k}}^{t_{k+1}}S_{h}(T-t_{k})P_{h}B(X_{h}^{k})P_{J} dW(s). \label{eq_SPDE_num_sol0}
\end{split}
\end{equation}
Over one step, this defines an alternative adaptive scheme \adptzero
\begin{equation}\label{eq:scheme_adapt0}
X_{h}^{n+1} = \varphi_{0}(-\Delta t_{n+1}A_{h})\left( X_{h}^{n} + P_{h}F(X_{h}^{n})\Delta t_{n+1} + P_{h}B(X_{h}^{n})P_{J} \Delta W_{n+1} \right), 
\end{equation}
where $\varphi_{0}(z):= e^{z}$.
\subsection{Adaptive time-stepping and admissible strategies}\label{sec:time_stepping}
In this section we define the adaptive time-stepping strategy and outline the conditions required for the strong convergence of the numerical approximation.
\begin{defn}\label{defn:Gn}
Let each member of the family $\{t_{n}\}_{n\in\mathbb{N}}$ be an $\mathcal{F}_{t}$ - stopping time, that is $\{t_{n} \leq t\} \in \mathcal{F}_{t}$ for all $t \geq 0$ with $(\mathcal{F}_{t})_{t\geq 0}$ as the natural filtration of $W$. We define the discrete-time filtration $\{\mathcal{G}_{n}\}_{n\in\mathbb{N}}$ as
\[\mathcal{G}_{n} := \{ \mathcal{B} \in \mathcal{F} : \mathcal{B} \cap  \{t_{n} \leq t\} \in \mathcal{F}_{t} \}, \quad n\in\mathbb{N}. \]
\end{defn}
\begin{assumption}\label{assumption_timestepping}
We assume each $\Delta t_{n} := t_{n}-t_{n-1}$ is $\mathcal{G}_{n-1}$-measurable and $N$ is a random integer such that
\[ N := \max \{n\in\mathbb{N} : t_{n-1} < T \} \text{ and }  t_{N} = T. \]
Further, there exists $\Delta t_{\min}$, $\Delta t_{\max}$ and $\rho\geq 1$ such that 
\begin{equation}
\Delta t_{\max} = \rho \Delta t_{\min}.
\end{equation} 
and $\Delta t_{\min}\leq \Delta t_{n} \leq \Delta t_{\max}$ for all $n$.
\end{assumption}
\begin{remark}\label{rmk:dtmeasurable}
The constraints on the size of $\Delta t_{n}$ ensures the numerical method terminates in a finite number of steps and to prove convergence as $\Delta t_{\max}\rightarrow 0$. 
By construction, the random time-step $\Delta t_{n+1}$ being $\mathcal{G}_{n}$-measurable ensures the Wiener increments $\Delta\beta_{j+1}:=\beta_{j}(t_{n+1})-\beta_{j}(t_{n})$ are $\mathcal{G}_{n}$-conditionally normally distributed with mean zero and variance $\Delta t_{n+1}$. That is
\begin{equation}
\Econd{\Delta\beta_{j+1}} = 0, \quad 
\Econd{\norm{\Delta\beta_{j+1}}^{2}} = \Delta t_{n+1}. \nonumber
\end{equation}
Without this construction, we may not assume the increments $\Delta \beta_{j}(t_{n+1})$ are normally distributed as the chosen step-size $\Delta t_{n+1}$ would depend on the solution value at $t_{n}$. 
\end{remark}
The final ingredient for the adaptive time-stepping strategy is the admissibility of the time-stepping strategy, which we now define as in \citep{kellylord2017adaptive}.

\begin{defn}\label{defn:addmissable}
Let $\{ X_{h}^{n} \}$ be an approximate numerical solution of \Cref{eq_SPDE_def}. The time-stepping of the numerical solution $\{\Delta t_{n}\}_{1\leq n\leq N}$ is called an \textbf{admissible time-stepping strategy} if \Cref{assumption_timestepping} is satisfied and there exists non-negative constants $R_{1}$ and $R_{2}$ such that
\begin{equation}
\norm{F(X_{h}^{n})}^{2} \leq R_{1} + R_{2}\norm{X_{h}^{n}}^{2}, \quad n = 0,1,...,N-1,
\end{equation}
and $\Delta t_{\min} < \Delta t_{n} \leq \Delta t_{\max}$.
\end{defn}

\Cref{lemma_timestep_selection} below gives several rules that enable the correct selection of the time-step size in order to satisfy our admissibility bound. The introduction of a minimum time-step size implies we cannot always ensure control of the drift response via time-step selection - without a potentially infinite number of time-steps. To construct a viable scheme we require a ``backstop scheme'' to run over a time-step of size $\Delta t_{\min}$ whenever \Cref{lemma_timestep_selection} selects a time-step size less than or equal to $\Delta t_{\min}$. The backstop may be any method that converges strongly and for our numerical simulations in \Cref{sec:numerics}, we use the so called nonlinearities-stopped exponential Euler of \citep{Jentzen2015}, which we refer to as \stp{} and define fully in \cref{eq:scheme_stopped} and \cref{eq:stopped:ind}. We can relate the admissible time-stepping strategy to the nonlinearities-stopped method in the following manner. Ignoring any potential super-linear growth in the diffusion, \stp{} uses the indicator function to ``switch off'' the drift if
\begin{equation}
\norm{F(X_{h}^{n})} > \frac{1}{\Delta t^{\theta}}, \label{eq_Jentzen_drift_cond}
\end{equation}
for some $\theta \in (0, 1/4]$. From \cref{eq_Jentzen_drift_cond}, we may expect that our adaptive method should satisfy
\begin{equation}
\Delta t_{n+1} \leq \frac{1}{\norm{F(X_{h}^{n})}^{1/\theta} }.\label{eq_Jentzen_drift_cond2}
\end{equation}
In fact, we show that \cref{eq_Jentzen_drift_cond2} is admissible for any $\theta \in \mathbb{R} \setminus \{0\}$ by \Cref{lemma_timestep_selection} below (and not just for
$\theta \in (0, 1/4]$). 

\Cref{lemma_timestep_selection} enables different time-step selection strategies to be employed in an adaptive scheme. By reducing the time-step size when the drift is large we maintain a similar control of the drift response as in \Cref{eq_Jentzen_drift_cond} without resorting to ``switching off'' the drift contribution. We expect this to lead to more accurate simulations of the SPDE as the non-linear contribution to the solution is never discounted for $\Delta t_{n}$ greater than $\Delta t_{\min}$. In general, if the adaptive scheme cannot set $\Delta t_{n}$ greater than $\Delta t_{\min}$, we must check the stopping condition, \Cref{eq_Jentzen_drift_cond}, to ensure overall convergence. 

A moderate choice of $\Delta t_{\min}$ and $\Delta t_{\max}$ can substantially increase the solution accuracy, as demonstrated numerically in \Cref{sec:numerics} for $\rho=100$ on two test problems. We cannot know \emph{a priori} that the adaptive scheme always satisfies \Cref{eq_Jentzen_drift_cond} at $\Delta t_{\min}$, however the additional buffer given by $\rho$ before the need to test against the stopping criterion at $\Delta t_{\min}$ ensures the adaptive scheme never has to discount the non-linear contributions for the numerical test cases.

\begin{lemma}\label{lemma_timestep_selection}
Let $\{X_{h}^{n}\}_{n \in \mathbb{N}}$ be an approximate numerical solution of \Cref{eq_SPDE_def}, let $\delta \leq \Delta t_{\max}$, $\theta \in \mathbb{R}\setminus \{0\}$ fixed and assume $\{\Delta t_{n}\}_{1\leq n\leq N}$ satisfies \Cref{assumption_timestepping}. Let $F$ satisfy \Cref{assump_F_local_lipschitz,assump_DF_growth}. Then $\{\Delta t_{n}\}_{1\leq n\leq N}$ is admissible if any one of the following hold:
\begin{enumerate}[i)]
\item $\Delta t_{n+1} \leq \delta /\norm{F(X_{h}^{n})}^{1/\theta}$
\item $\Delta t_{n+1} \leq \delta / (c_{3}(1 + \norm{X_{h}^{n}}^{1+c_{2}} ))$
\item $\Delta t_{n+1} \leq \delta\norm{X_{h}^{n}} /\norm{F(X_{h}^{n})}$
\item $\Delta t_{n+1} \leq \delta\norm{X_{h}^{n}} / (c_{3}(1 + \norm{X_{h}^{n}}^{1+c_{2}} ))$
\item $\Delta t_{n+1} \leq \delta / \norm{X_{h}^{n}}$.
\end{enumerate}

\end{lemma}
\begin{proof}
For part (i), applying the definitions of $\Delta t_{\min}$ and $\Delta t_{\max}$,
\begin{equation}
\norm{F(X_{h}^{n})}^{2} 
\leq \left( \frac{\delta}{\Delta t_{n+1}} \right)^{2\theta} 
\leq \left( \frac{\Delta t_{\max}}{\Delta t_{\min}} \right)^{2\theta}
= \rho^{2\theta} \nonumber
\end{equation}
which is admissible. For parts (ii)-(iv) see \citep{kellylord2017adaptive}. Part (v) is seen by application of \Cref{assump_DF_growth} and a similar argument as for (i).
\end{proof}
The specific time-step selection rule and the value of $\rho$ that gives the best trade-off between accuracy and computational cost is unknown and may be problem dependent. One point to note for implementation is that admissibility only requires one of the conditions in \Cref{lemma_timestep_selection} to hold. Therefore we may compute several, or all, of the timestep restriction criteria and then set $\Delta t_{n+1}$ to the maximum timestep size to minimise the overall computational cost of the scheme. 

\section{Main Results}\label{sec:mainresults}

We now state our main result on strong convergence of the adaptive time-stepping method. For all results in this section we require that \Cref{assump_A_semigroup_generator,assump_F_local_lipschitz,assump_DF_growth,assump_B_lip,assump_B_growth}
hold. We also assume the time-stepping strategy $\{\Delta t_{n}\}_{1\leq n\leq N}$ satisfies \Cref{defn:addmissable} and $\Delta t_{n+1}>\Delta t_{\min}$ for all $0\leq n<N$. The constant $\Cmain$ depends on the final time $T$, the constant $\rho$ and the initial data. It is independent of the spatial discretisation parameter $h$, the maximum time-step size $\Delta t_{\max}$, and the number of Fourier modes, $J$, taken in the approximation of $\Delta W$. As in \Cref{sec:assumptions} the eigenvalues of $A$ are denoted by $\{ \lambda_{j} \}_{j=1}^{\infty}$.

Firstly for the semi-discrete approximation in space, we state the following strong convergence result.
\begin{theorem}\label{thm:Eh}
Let $X(T)$ be the mild solution defined in \Cref{eq:mild_sol} at time $T$. Let $X_{h}(T)$ be the semi-discrete approximation defined in \Cref{eq:SPDE_SD_sol}. Then, for $X_{0} \in L^{2}(\mathbb{D},\DAs{1/2})$ and $\epsilon>0$, the following estimate holds
\begin{equation}
\Lnorm{X(T)-X_{h}(T) } 
\leq 
\Cspatial
\left(
h^{1+r-\epsilon} + 
\lambda_{J+1}^{-\frac{1+r}{2}+\epsilon}
\right). \nonumber 
\end{equation}
\end{theorem}

Next, we state the strong convergence result for the full space-time discretisation.
\begin{theorem}\label{thm:Ek}
Let $X(T)$ be the mild solution defined in \Cref{eq:mild_sol} at time $T$. Let $X_{h}^{N}$ be the numerical approximation \adptone{} defined in \Cref{eq:scheme_adapt1} with $\{t_{n}\}_{n\in\mathbb{N}}$ an admissible time-stepping strategy. Then for $X_{0} \in L^{2}(\mathbb{D},\DAs{1/2})$ and $\epsilon>0$, the following estimate holds
\begin{equation}
\Lnorm{X(T)-X_{h}^{N} } 
\leq 
\Cmain
\left(
h^{1+r-\epsilon} + 
\lambda_{J+1}^{-\frac{1+r}{2}+\epsilon} +
\Delta t_{\max}^{ \frac{1}{2}-\epsilon  }
\right). \nonumber 
\end{equation}
\end{theorem}

The final result shows strong convergence, of the same order, when a lower order of exponential approximation is used to solve the non-linear integral in the space-time discretisation.
\begin{corollary}\label{corr:phi0}
Let $X(T)$ be the mild solution defined in \Cref{eq:mild_sol} at time $T$. Let $X_{h}^{N}$ be the numerical approximation \adptzero{} defined in \Cref{eq:scheme_adapt0} with $\{t_{n}\}_{n\in\mathbb{N}}$ an admissible time-stepping strategy. Then for $X_{0} \in L^{2}(\mathbb{D},\DAs{1/2})$ and $\epsilon>0$, the following estimate holds
\begin{equation}
\Lnorm{X(T)-X_{h}^{N} } 
\leq 
\Cmain
\left(
h^{1+r-\epsilon} + 
\lambda_{J+1}^{-\frac{1+r}{2}+\epsilon} +
\Delta t_{\max}^{ \frac{1}{2}-\epsilon  }
\right). \nonumber 
\end{equation}
\end{corollary}

The strategy of the proofs are the following. Firstly, due to adaptive time-stepping, we  must derive a conditional \ito{} isometry and hence conditional regularity in time for the mild solution in \Cref{sec:cond_ito}. Then in \Cref{sec:proofmain} we first prove the spatial convergence, and secondly the full spatio-temporal convergence.

\section{Numerical Experiments}\label{sec:numerics}

We continue with several numerical tests to demonstrate the effectiveness of the adaptive time-stepping scheme in comparison to alternative numerical schemes that ensure strong convergence in the presence of non-globally Lipschitz drift. The numerical experiments all utilise the spectral Galerkin approach. We firstly restate our adaptive numerical scheme and introduce several additional schemes to compare against. The first adaptive numerical scheme, \adptone{}, is defined by 
\begin{align}\label{eq:APT1}
X_{h}^{n+1} &= e^{-\Delta t_{n+1}A_{h}}\left( X_{h}^{n}  + P_{h}B(X_{h}^{n})P_{J} \Delta W_{n+1} \right) \nonumber \\
&\breakgap + 
\Delta t_{n+1} \varphi_{1}(\Delta t_{n+1}A_{h}) P_{h}F(X_{h}^{n}), 
\end{align}
whenever $\Delta t_{n+1} > \Delta t_{\min}$. The operator $P_{J}$ is defined in \cref{eq:PJ} and for all schemes we set $J=N_{x}$, the number of nodes in the spatial domain. 

A less computationally intensive adaptive scheme, \adptzero, can be defined by using the $\varphi_{0}$ approximation to the semigroup in the drift,
\begin{equation}\label{eq:APT0}
X_{h}^{n+1} = e^{-\Delta t_{n+1}A_{h}}\left( X_{h}^{n} + P_{h}F(X_{h}^{n})\Delta t_{n+1} + P_{h}B(X_{h}^{n})P_{J} \Delta W_{n+1} \right),
\end{equation}
whenever $\Delta t_{n+1} > \Delta t_{\min}$. Both adaptive schemes use \cref{lemma_timestep_selection} iii) to (maximally) choose $\Delta t_{n+1}$ at each time-step and \stp{} if $\Delta t_{n+1} = \Delta t_{\min}$. \stp{} is defined by
\begin{equation}\label{eq:scheme_stopped}
X_{h}^{n+1} = e^{-\Delta tA_{h}}X_{h}^{n} + 
e^{-\Delta tA_{h}}\left[ P_{h}F(X_{h}^{n})\Delta t + P_{h}B(X_{h}^{n})P_{J} \Delta W_{n+1} \right]_{\mathbbm{1}\{F,\Delta t\} },
\end{equation}
with
\begin{equation}\label{eq:stopped:ind}
\mathbbm{1}\{F,\Delta t\}:= 
\mathbbm{1}\Big\{ \norm{P_{h}F(X_{h}^{n})} \leq (1/\Delta t)^{\theta} \Big\}.
\end{equation}
We take $\theta$ as large as possible (according to \citep{Jentzen2015}), which corresponds to $\theta=\frac{1}{4}$. In this method, the nonlinear contributions to the scheme are switched off if they are too large in comparison to the timestep size. This controls any unwanted growth due to the non-Lipschitz drift and ensures strong convergence of the scheme. The Lipschitz assumptions in \citep{Jentzen2015} are more general than in this work, as they allow (potentially) non-globally Lipschitz drift $B$ to be controlled by the one-sided Lipschitz drift $F$ in a combined estimate such as
\begin{equation}
\langle F(X)-F(Y),X-Y \rangle +
\norm{ B(X)-B(Y) }_{L_{0}^{2}}
\leq L \norm{X-Y}^{2} \nonumber .
\end{equation}
As this work only considers Lipschitz drift $B$, we have relaxed the indicator bound to  \Cref{eq:stopped:ind} instead of the stronger bound of
\begin{equation}
\mathbbm{1}\{F,\Delta t\}:= 
\mathbbm{1}\Big\{ \norm{P_{h}F(X_{h}^{n})} + \HSnorm{P_{h}B(X_{h}^{n})}{H}{H}\leq (1/\Delta t)^{\theta} \Big\},
\end{equation}
as originally defined in \citep{Jentzen2015}, to improve the performance of the stopped scheme. 

The second fixed step numerical scheme we compare against is the tamed Euler-Maryuama introduced in \citep{Gyongy2016}. This work is of the variational approach and therefore  satisfies different assumptions than in this paper. In \citep{Gyongy2016} the linear and non-linear operators are considered as a single operator which consists of two components. These components satisfy different growth and coercivity conditions than in this work. Existence and uniqueness of solutions as well convergence of a tamed Euler-Maryuama scheme is proven in \citep{Gyongy2016}. We test a similar tamed Euler-Maryuama method but again remark the framework and assumptions of \citep{Gyongy2016} are quite different to what we consider here. To that end, define $C(X):= -AX + F(X)$, where $A$ and $F$ are the linear and nonlinear parts of the SPDE. Then the tamed operator $\tilde{C}$ is defined by
\begin{equation}\label{eq:taming}
\tilde{C}:= \frac{C(X)}{1+\sqrt{\Delta t}\norm{C(X)}},
\end{equation}
and the tamed scheme as, 
\begin{equation}\label{eq:scheme_tamed}
X_{h}^{n+1} = X_{h}^{n} + \Delta t P_{h}\tilde{C}(X_{h}^{n}) + P_{h}B(X_{h}^{n})P_{J} \Delta W_{n+1},
\end{equation}
called \tamed{}. The final method we compare against, named \tamedexp{}, is a speculative method, inspired by the tamed schemes of e.g. \citep{Gyongy2016} and \citep{hutzenthaler_tamed2012}. We define the tamed exponential scheme as
\begin{equation}\label{eq:scheme_exp_tamed}
X_{h}^{n+1} = e^{-\Delta t A_{h}}\left( X_{h}^{n} + P_{h}\tilde{F}(X_{h}^{n})\Delta t + P_{h}B(X_{h}^{n})P_{J} \Delta W_{n+1} \right),
\end{equation}
with $\tilde{F}$ defined analogously to \Cref{eq:taming} for $F$ in place of $C$. 

To ensure that the linear operator is the generator of an analytic semigroup with eigenvalues $\lambda_{j}>0$, we add and subtract $c_{0}X dt$ on the right hand side of \Cref{eq_SPDE_def} for some $c_{0}\in \mathbb{R}$, taking $-c_{0}X$ into the (re)definition of $-A$ and $c_{0}X$ into $F$ for the exponential methods. In all simulations we set $c_{0}= \max\{-\lambda_{j},0\}+1$. In the case of $-A$ equal to the Laplacian, this means we set $c_{0}=1$ as all $\lambda_{j}\geq 0$.

In all simulations, we use the time-stepping rule $\Delta t_{n} \leq \frac{\delta}{\norm{F(X_{h}^{n})}}$, set $\rho=100$ and $N_{x}=512$. For the convergence and efficiency plots we solve over a range of $\Delta t_{\max} \in [2^{-3}, 2^{-12}]$ with $1000$ trials. The reference solutions are computed using \stp{} at $\Delta t_{\text{ref}}=5\cdot 10^{-6}$.

\subsection{Allen-Cahn Equation}
We test the adaptive scheme firstly on the Allen-Cahn SPDE with multiplicative noise, that is we solve the following SPDE
\begin{equation}
dX = \left[ \Delta X + X-X^{3} \right] dt + \sigma X dW, \nonumber
\end{equation}
on the 1D domain $[0,32\pi]$. We set $\sigma=1$ and compare the mean-square error over 1000 trials of the final solution at $T=1$, against a reference solution. In \Cref{fig:AC_dt_converge} we plot the mean-square error against $\Delta t_{\max}$ for a range of noise regularities $r\in\{0,\frac{1}{2},1\}$, and observe convergence at rates equal to the theoretical rate of $\frac{1}{2}$. Note that the convergence results only hold for $r\in [0,1)$ and not for $r=1$ due to blow up for a constant in \cref{eq:spatial:II2} as $r\rightarrow 1$, however this is not observed in the either of the numerical tests.
\begin{figure}
\centering
\begin{subfigure}{.5\textwidth}
  \centering
  \includegraphics[width=.9\linewidth]{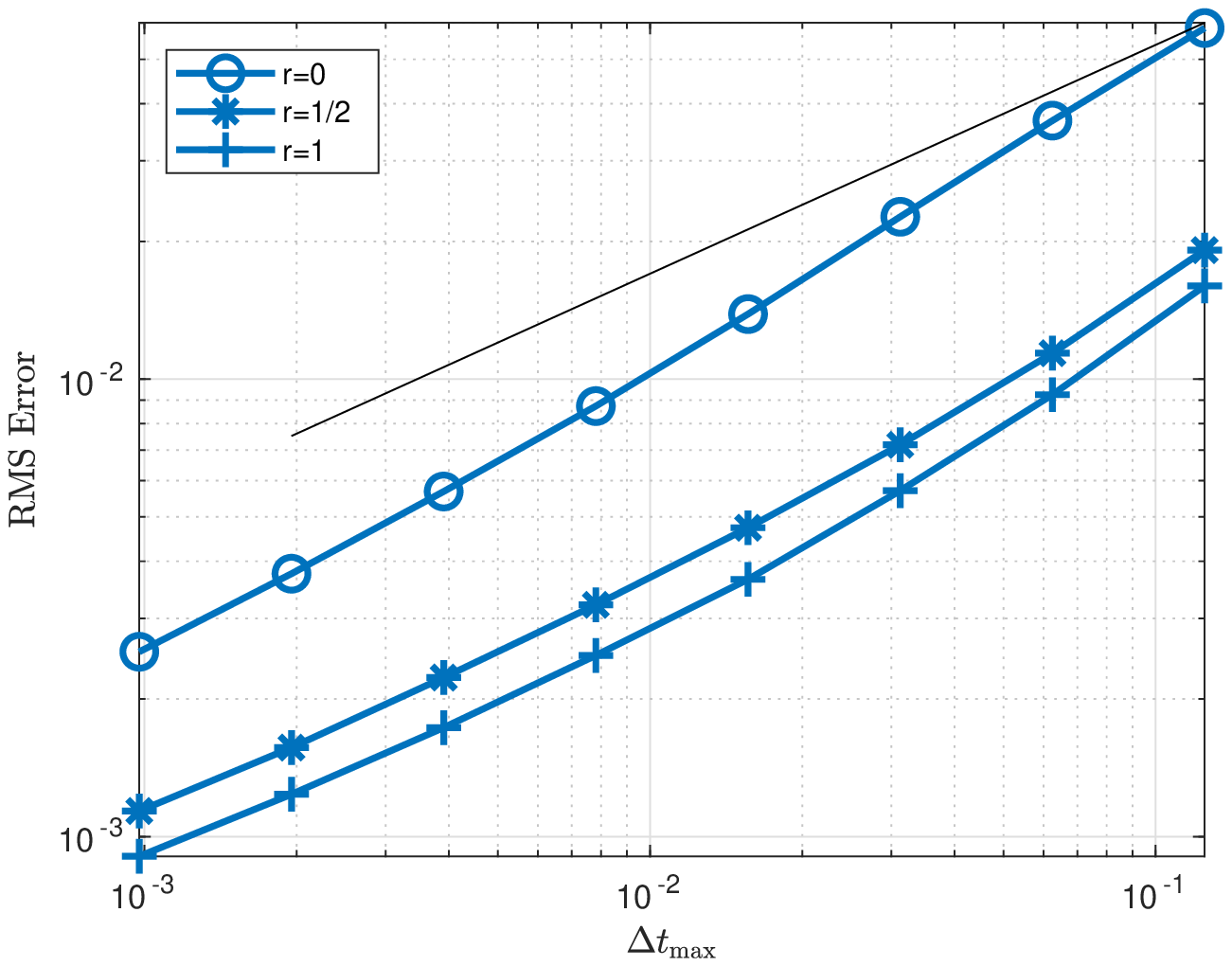}
  \caption{ }
  \label{fig:AC_dt_converge}
\end{subfigure}%
\begin{subfigure}{.5\textwidth}
  \centering
  \includegraphics[width=.9\linewidth]{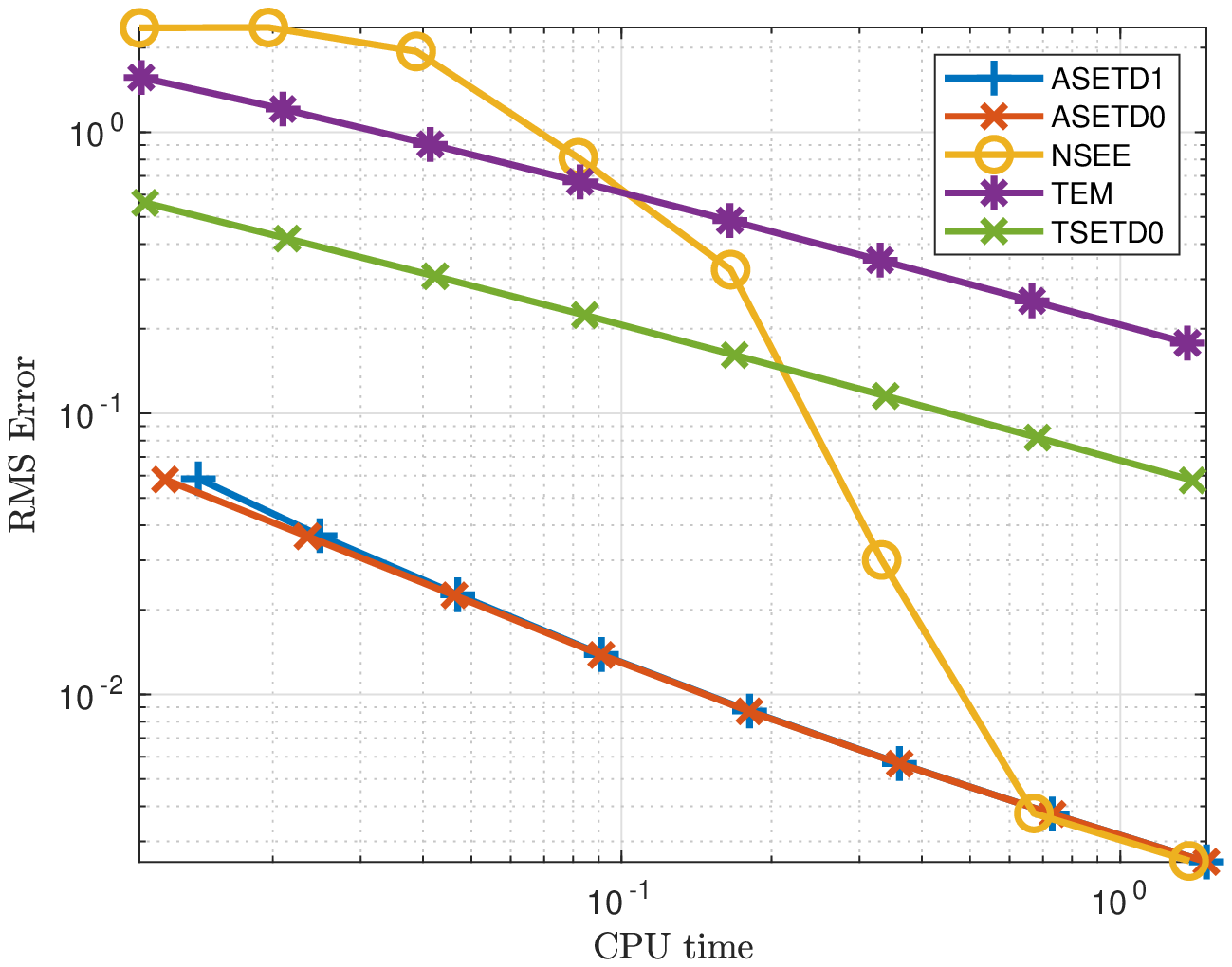}
  \caption{ }
  \label{fig:AC_err_vs_cpu_r0}
\end{subfigure}
\caption{Plot (a) convergence of \adptone{} with respect to $\Delta t_{\max}$. We show convergence for a range of noise regularity parameters, $r\in\{0,\frac{1}{2},1\}$ with reference slope of order $\frac{1}{2}$. Plot (b) shows the RMS error against average CPU time for all schemes with noise regularity $r=0$ and $\rho=100$.}
\end{figure}

In \Cref{fig:AC_err_vs_cpu_r0} we compare the efficiency of two adaptive schemes, \adptone{} and \adptzero{}, against the three fixed time-stepping methods \stp{}, \tamed{}  and \tamedexp{}. We observe significantly lower errors and a higher convergence rate for \adptone{} and \adptzero{} when compared to both \tamed{} and \tamedexp{}. \stp{} cannot produce accurate solutions for larger values of $\Delta t$, due to the switching off of the non-linear contributions. After sufficient reduction in $\Delta t$, the errors in \stp{} and \adptzero{} coincide, with a slight overhead for the adaptive scheme. For the parameter $r=0$ we do not see significant difference between \adptone{} and \adptzero{}, however for smoother noise, such as shown in \Cref{fig:SH_err_vs_cpu_r0}, the increase of accuracy of \adptone{} becomes apparent. A final point to note is the difference in accuracy between \tamed{} and \tamedexp{} for $r=0$, for smoother noise (larger value of $r$) the accuracy of these two schemes coincide.

In \Cref{fig:AC_change_in_dt} we plot, for a single representative realisation of the Allen-Cahn equation, the change in $\Delta t_{n}$ over the time period $[0,5]$ in the upper plot. Below this we show the change in $\norm{F(X)}$ over time (we observe close agreement between \adptone{} and the reference solution, which is not shown for clarity). The buffer between $\Delta t_{\max}$ and $\Delta t_{\min}$ provided by $\rho$ enables \adptone{} to approximate the solution without use of the backstop at any point for this realisation. Conversely, for this particular realisation, \stp{} has to discard the non-linear contribution for significant portions of the overall time period as seen by the orange dashed line and grey shaded areas. In this realisation the parameters used were $r=0$, $\Delta t_{\max}=10^{-2}$ and $\rho=100$. The resulting mean timestep was $\overline{\Delta t}=3\cdot 10^{-3}$.
\begin{figure}
\centering
\includegraphics[width=.8\linewidth]{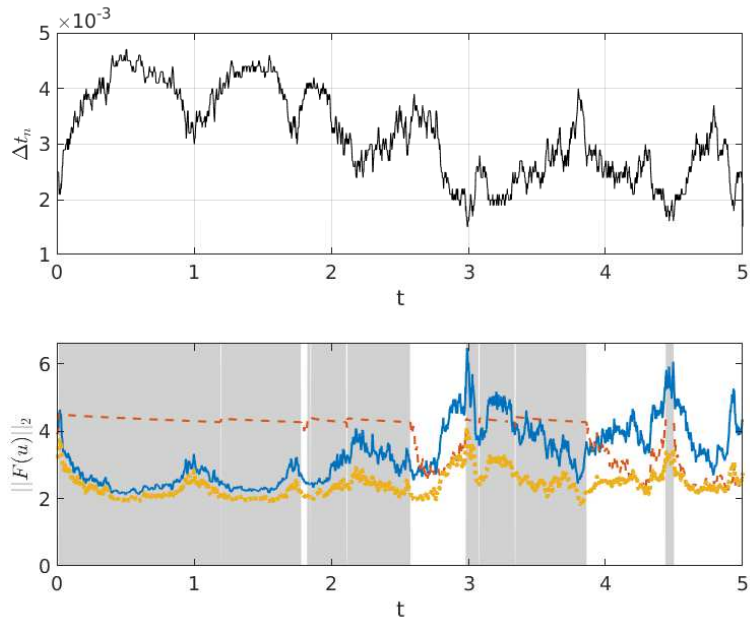}
\caption{Upper plot: change in $\Delta t_{n}$ over time for a single solve of Allen-Cahn equation. Lower plot: size of the non-linearity $\norm{F(X)}$ for the same realisation. The solid blue line corresponding to \adptone{}, dashed orange for \stp{} and dotted yellow for \tamedexp{}. The shaded grey area indicates the times the stopped method excluded the non-linear contributions.}
\label{fig:AC_change_in_dt}
\end{figure}

\subsection{Swift-Hohenberg equation}

We consider similar numerical experiments as in the previous section, with the Swift-Hohenberg SPDE \citep{Cross1993}. This is a $4^{th}$ order SPDE that arises in applications involving pattern formation such as fluid flow and the study of neural tissue. It is defined as
\begin{equation}
dX = \left[ \eta X - (1+\Delta)^{2}X + cX^{2}-X^{3} \right] dt + \sigma X dW \nonumber.
\end{equation}

We set the parameters $\eta = −0.7, c = 1.8$ and $\sigma = 1$, then solve to a final time of $T=5$ with $N_{x}=512$ points in the spatial domain $[0,32\pi]$. We plot example solutions in \Cref{fig:SH_solution_top} for all schemes. We clearly see the difference between \adptone{}, which faithfully recreates the patterns formed by the reference solution, and \stp{} which is unable to do so for this realisation. Again this effect is due to the size of $\Delta t$ being too large when compared to $\norm{F(X)}$ and after sufficient reduction in $\Delta t$ we would of course see the correct pattern appearing for the stopped method.

\begin{figure}
\centering
\includegraphics[width=.9\linewidth]{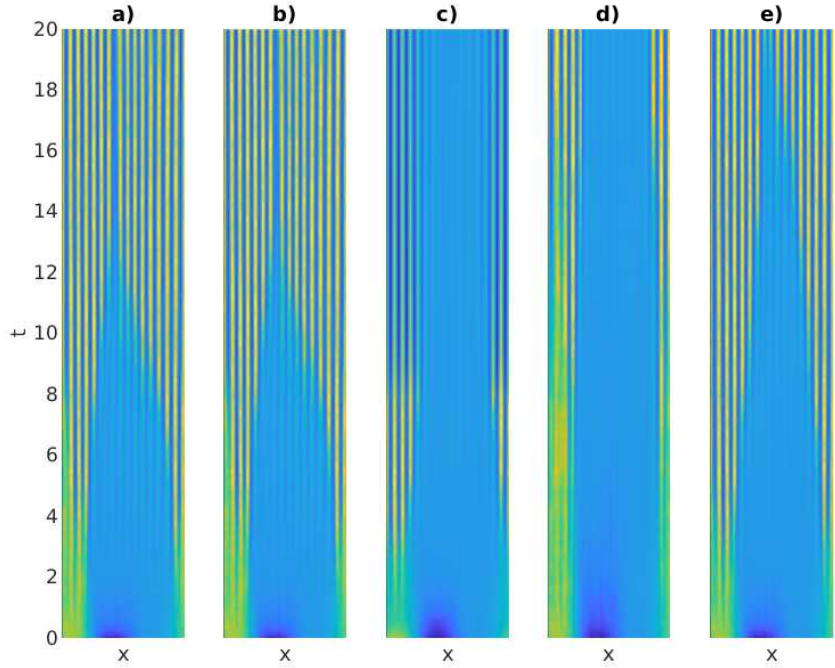}
\caption{Example solution of the Swift-Hohenberg equation using: a) Reference solution, b) \adptone{}, c) \stp{}, d) \tamed{} and e) \tamedexp{}. For all solutions parameters used $\Delta t_{\max}=10^{-2}$, $r=0$ and $\rho=100$. For fixed step methods $\overline{\Delta t}=1.5\cdot 10^{-3}$.}
\label{fig:SH_solution_top}
\end{figure}
\begin{figure}
\centering
\begin{subfigure}{.5\textwidth}
  \centering
  \includegraphics[width=.9\linewidth]{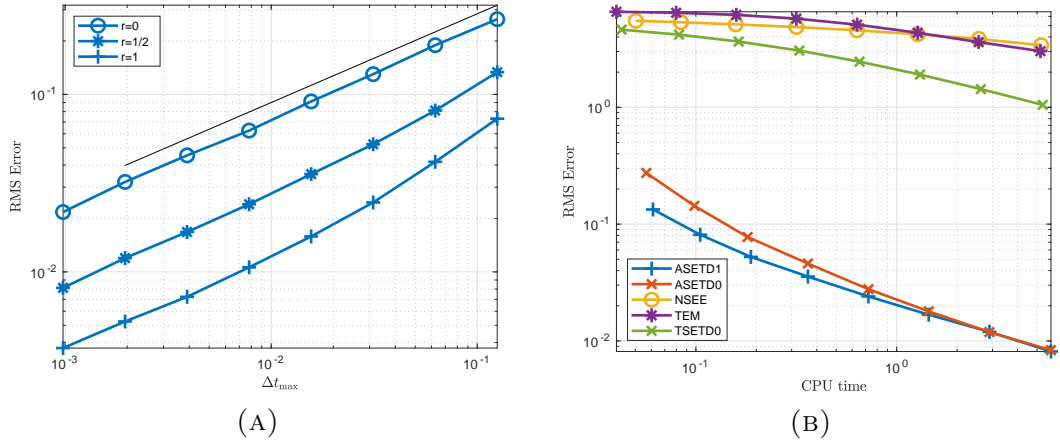}
  \caption{ }
  \label{fig:SH_dt_converge}
\end{subfigure}%
\begin{subfigure}{.5\textwidth}
  \centering
  \includegraphics[width=.9\linewidth]{{{./SH_err_vs_CPU_r0.5}}}
  \caption{ }
  \label{fig:SH_err_vs_cpu_r0}
\end{subfigure}
\caption{Plot a) convergence of \adptone{} with respect to $\Delta t_{\max}$. We show convergence for a range of noise regularity parameters, $r\in\{0,\frac{1}{2},1\}$ with reference slope of order $\frac{1}{2}$. Plot b) shows the RMS error against average CPU time for all schemes and noise regularity $r=\frac{1}{2}$.}
\end{figure}
\Cref{fig:SH_err_vs_cpu_r0} shows the RMS error against CPU time for the various methods and $r=\frac{1}{2}$. The smoother noise shows the gain in accuracy of \adptone{} over \adptzero{} at larger values of $\Delta t_{\max}$. It is clear that \stp{} and \tamed{} suffer at these values of $\overline{\Delta t}$ and cannot accurately solve the $4^{th}$ order SPDE. \adptone{} and \adptzero{} produce solutions with significantly lower errors, even when $\Delta t_{\max}$ is very large. We observe the expected convergence rate for \adptone{} in \Cref{fig:SH_dt_converge} over the entire range of $\Delta t_{\max}$. In comparison, the fixed step methods are only starting to reach their asymptotic convergence rate at the smallest values of $\Delta t$. We remark that after reducing $\overline{\Delta t}$ to a sufficiently small value, we would see the same behaviour as in \Cref{fig:AC_err_vs_cpu_r0}, where the error in \stp{} jumps down and agrees with \adptzero{}. However it should be noted that these experiments were carried out for 1D spectral Galerkin and thus are relatively inexpensive to run. For a large finite element discretisation in higher dimensions, it may be difficult to reach a sufficiently small $\overline{\Delta t}$ to ensure \stp{} always includes the (important) non-linear contributions. 

\section{Conditional \ito{} Isometry and Regularity of Mild Solution}\label{sec:cond_ito}
In this section, we prove a conditional \ito{} isometry for the $Q$-Wiener process and prove a conditional regularity result for the mild solution, which are required for the analysis of the numerical method. 
\begin{lemma}\label{thm_cond_ito_isometry}
Let $B: U_{0} \rightarrow H$ be an $L_{0}^{2}$-valued predictable process such that
\begin{equation}
\int_{0}^{T} \mathbb{E}\Big{[} \Bnorm{B(s)}^{2} \Big{]} ds < \infty, \nonumber
\end{equation}
and let $\mu, \tau$ be two stopping times with $0\leq \mu \leq \tau \leq T$. Then the following identities hold,
\begin{align}
\Phi:=\EcondF{
\norm{ \int_{\mu}^{\tau} B(s) dW(s) }^{2} }
&= 
\EcondF{\int_{\mu}^{\tau} \Bnorm{B(s)}^{2} ds }\label{eq_ito_isometry_SPDE} \\
\EcondF{\int_{\mu}^{\tau} B(s) dW(s) } &=0. \label{eq_ito_zero_mean_SPDE}
\end{align}
\end{lemma}

\begin{proof}
The proof is an expansion of the stochastic term $\Phi$ as an inner product in a basis of $H$, similarly to Theorem 10.16 in \citep{lord2014introduction}. Then an application of Theorem 5.21 and Lemma 5.22 from \citep{mao2008stochastic} to the individual terms of the expansion yields the result.
\end{proof}
We now prove a conditional regularity in time result for the mild solution, similar to standard regularity results for Lipschitz drift but of lower order. Conditional expectations are required throughout the analysis due to \Cref{rmk:dtmeasurable}.

\begin{lemma}\label{lemma_cond_regularity_time}
Let $X$ be the mild solution given by \Cref{eq_SPDE_mild_sol} and that \Cref{assump_A_semigroup_generator,assump_F_local_lipschitz,assump_DF_growth,assump_B_lip,assump_B_growth} hold. Let $t_{n},t_{n+1}\in [0,T]$ be such that $\Delta t_{n+1} := t_{n+1}-t_{n} > 0$ and $\epsilon>0$. The following conditional regularity estimates holds for an a.s. finite, $\mathcal{G}_{n}$-measurable random variable $\Kreg$, finite constant $\tilde{\Kreg}$ and constant $p\geq 1$
\begin{equation}
\Lcondnorm{X(t_{n+1})-X(t_{n})}^{2}
\leq 
\Kreg \quad a.s,
\quad \supT \mathbb{E}[K_{1}^{p}]\leq \Konep{p} \Delta t_{n+1}^{(1-\epsilon)p}.
\end{equation}
\end{lemma}
\begin{proof}
The solution $X(t_{n}) \in \DAs{\frac{\eta}{2} }$ for $\eta \in [0,1)$ by \Cref{eq_moment_bounds}. We can decompose the error terms between any two elements $u,v\in \DAs{\frac{\eta}{2} }$ using \Cref{lemma_Stv-v}, in particular \Cref{eq_I-St_v_bnd} and the triangle inequality as follows
\begin{align}
\norm{u-v}^{2} 
&= 
\norm{u-S(\Delta t_{n+1})v+S(\Delta t_{n+1})v-v}^{2} \nonumber \\
&\leq  
2\norm{u-S(\Delta t_{n+1})v}^{2} + 2\norm{(I-S(\Delta t_{n+1}))v}^{2} \nonumber \\
&\leq 
2\norm{u-S(\Delta t_{n+1})v}^{2} + 2\Cs\Delta t_{n+1}^{\eta}\norm{v}^{2}_{\eta}. \nonumber 
\end{align}
Setting $u=X(t_{n+1})$ and $v=X(t_{n})$ we have,
\begin{align}
\norm{X(t_{n+1})-X(t_{n})}^{2} 
&\leq 
2(\Cs\Delta t_{n+1}^{\eta}\norm{X(t_{n})}^{2}_{\eta} \nonumber \\
&\breakgap + 
\norm{X(t_{n+1})-S(\Delta t_{n+1})X(t_{n})}^{2}) \nonumber \\
&\leq 
\bigg{(} 2\Cs\Delta t_{n+1}^{\eta} \norm{X(t_{n})}^{2}_{\eta} \nonumber\\
&\breakgap+ 
4\norm{\int_{t_{n}}^{t_{n+1}} S(t_{n+1}-s)F(X(s)) ds }^{2} \nonumber\\
&\breakgap+ 
4\norm{\int_{t_{n}}^{t_{n+1}} S(t_{n+1}-s)B(X(s))dW(s)}^{2}\bigg{)} \nonumber\\
&:=
\Big{(} 2\norm{I}^{2} + 4\norm{II}^{2} + 4\norm{III}^{2}\Big{)} .\label{eq_reg_time_parts}
\end{align}
The first term is directly bound by \Cref{eq_moment_bounds} with $\eta = 1-\epsilon$, $\epsilon>0$,
\begin{align}
\Lcondnorm{I}^{2}
&\leq 
\Cs\Delta t_{n+1}^{1-\epsilon}
\Econd{ \norm{X(t_{n})}^{2}_{1-\epsilon} }. \label{eq:regI}
\end{align}
For part $II$, we apply Jensen's inequality, the boundedness of the semigroup, \Cref{eq_F_growth} and \Cref{thm:moment_bounds}
\begin{align}
\Lcondnorm{II}^{2}
&= 
\Lcondnorm{\int_{t_{n}}^{t_{n+1}} S(t_{n+1}-s) F(X(s)) ds}^{2} \nonumber\\
&\leq 
\Delta t_{n+1}\int_{t_{n}}^{t_{n+1}}
\Lcondnorm{S(t_{n+1}-s) F(X(s))}^{2} ds \nonumber\\
&\leq
\Delta t_{n+1}\int_{t_{n}}^{t_{n+1}}
\Lcondnorm{F(X(s))}^{2} ds \nonumber\\
&\leq
2c_{3}^{2}\Delta t_{n+1}
\Econd{ \int_{t_{n}}^{t_{n+1}}
(1+\norm{X(s)}^{2c_{2}+2}) ds
}. \label{eq:regII}
\end{align}
For part $III$ we show, via \Cref{thm_cond_ito_isometry}, boundedness of the semigroup, \Cref{assump_B_growth} and \Cref{thm:moment_bounds}, that
\begin{align}
\Lcondnorm{III}^{2}
&= 
\Lcondnorm{\int_{t_{n}}^{t_{n+1}} S(t_{n+1}-s) B(X(s)) dW(s)}^{2} \nonumber\\
&= 
\Econd{
\int_{t_{n}}^{t_{n+1}} 
\Bnorm{S(t_{n+1}-s) B(X(s))}^{2} ds
}\nonumber\\
&\leq 
2\LBr^{2} \Econd{\int_{t_{n}}^{t_{n+1}} ( 1 + \norm{X(s)}^{2} ) ds }. \label{eq:regIII}
\end{align}
We can now define the a.s. finite $\mathcal{G}_{n}$-measurable random variable $\Kreg$ as
\begin{align}
\Kreg
:=
2\left(
\Cs\ \Delta t_{n+1}^{1-\epsilon }\Econd{ \norm{X(t_{n})}^{2}_{1-\epsilon} } \right. 
&+ 8c_{3}^{2}\Delta t_{n+1}\Econd{\int_{t_{n}}^{t_{n+1}} (1+\norm{X(s)}^{2c_{2}+2}) ds}  \nonumber \\
& \left. + 8\LBr^{2} \Econd{\int_{t_{n}}^{t_{n+1}} (1 + \norm{X(s)}^{2}) ds} \right).\nonumber
\end{align}
Together, \Cref{eq:regI,eq:regII,eq:regIII} imply
\begin{equation}
\Lcondnorm{X(t_{n+1})-X(t_{n})}^{2}
\leq 
\Kreg \quad a.s. \nonumber
\end{equation}
Further, we have for $p\geq 1$
\begin{align}
\supT \mathbb{E} \left[ \Kreg^{p} \right] 
&\leq
2.3^{p}\Cs^{p} \Delta t_{n+1}^{(1-\epsilon)p} \supT \mathbb{E}\left[ \norm{X(s)}^{2p}_{1-\epsilon} \right]
+
8.3^{p}c_{3}^{2p}\Delta t_{n+1}^{2p} \supT \mathbb{E}\left[ 1+\norm{X(s)}^{(2c_{2}+2)p} \right] \nonumber \\
&+
8.3^{p}\LBr^{2p}\Delta t_{n+1}^{p} \supT \mathbb{E}\left[ 1 + \norm{X(s)}^{2p} \right] \nonumber \\
&=:
\Delta t_{n+1}^{(1-\epsilon)p} \Konep{p}
\end{align}

and \Cref{thm:moment_bounds} implies that $\Konep{p}<\infty$.
\end{proof}

For the analysis of the numerical scheme, we require conditional bounds on the remainder terms when we Taylor expand the drift and diffusion terms.

\begin{lemma}\label{lemma_remainder_terms}
Define the remainder terms of $F$ and $B$ from the Taylor expansions as follows,
\begin{align}
F(X(s)) &= F(X(t_{n})) + R_{F}(s,t_{n},X(t_{n}) ) \nonumber \\
B(X(s)) &= B(X(t_{n})) + R_{B}(s,t_{n},X(t_{n}) ), \nonumber
\end{align}
with $R_{F}$ and $R_{B}$ given by
\begin{equation}
R_{z}(s,t_{n},X(t_{n}) ) : = 
\int_{0}^{1} Dz
\left( X(t_{n}) + \tau ( X(s)-X(t_{n}) ) \right) 
( X(s)-X(t_{n}) ) d\tau, \nonumber 
\end{equation}
for $z$ as either $F$ or $B$. Then the following bounds hold for a.s. finite $\mathcal{G}_{n}$-measurable random variables $\KF, \KFsq, \KB$ and finite constants $\tilde{\KF}, \tilde{\KFsq}, \tilde{\KB}$
\begin{align*}
i) \quad 
\Econd{ \norm{ \int_{t_{n}}^{t_{n+1}}R_{F}(s,t_{n},X(t_{n}) )  ds } }
&\leq 
\KF \Delta t_{n+1} \quad a.s,  \nonumber \\
 \supTE{\KF\Delta t_{n+1} }  &\leq \tilde{\KF}\Delta t_{n+1}^{\frac{3}{2}-\epsilon}, \nonumber \\
ii) \quad 
\Econd{ \norm{ \int_{t_{n}}^{t_{n+1}}R_{F}(s,t_{n},X(t_{n}) ) ds }^{2} }
&\leq 
\KFsq \Delta t_{n+1} \quad a.s,  \nonumber \\
  \supTE{\KFsq\Delta t_{n+1} }  &\leq \tilde{\KFsq}\Delta t_{n+1}^{3-\epsilon}, \nonumber \\
iii) \quad 
\Econd{ \norm{ \int_{t_{n}}^{t_{n+1}}R_{B}(s,t_{n},X(t_{n}) ) dW(s) }^{2} }
&\leq 
\KB \Delta t_{n+1} \quad a.s, \\
  \supTE{\KB\Delta t_{n+1} }  &\leq \tilde{\KB}\Delta t_{n+1}^{2-\epsilon}. \nonumber 
\end{align*}
\end{lemma}

\begin{proof}
By \Cref{assump_DF_growth} on $F$ and the Cauchy-Schwarz inequality we have 
\begin{align}
&\Econd{\norm{ R_{F}(s,t_{n},X(t_{n}))} } \nonumber \\
&\leq
\Econd{
\int_{0}^{1}c_{1}(1 + \norm{ X(t_{n}) + \tau (X(s)-X(t_{n}) ) }^{c_{2}})\norm{X(s)-X(t_{n})} d\tau  
} \nonumber \\
&\leq
c_{1} \Lcondnorm{(X(s)-X(t_{n})} \nonumber \\
&\breakgap \times \sqrt{\Econd{
\int_{0}^{1}c_{1}(1 + \norm{X(t_{n}) + \tau (X(s)-X(t_{n}) )}^{c_{2}})^{2}d\tau 
}}. \nonumber 
\end{align}
We can then use \Cref{lemma_cond_regularity_time} to bound the first part by $\sqrt{\Kreg}$ and define an a.s. finite $\mathcal{G}_{n}$-measurable random variable 
\begin{equation}
M_{1} := \Econd{2c_{1}^{2}(1 + 3^{2c_{2}} \norm{X(s)}^{2c_{2}} + 2.3^{2c_{2}}\norm{X(t_{n})}^{2c_{2}} ) },\nonumber 
\end{equation}
to bound the second term. We define $\tilde{M_{1}}$ as 
\begin{equation}
\supT \mathbb{E}\left[ M_{1} \right] 
\leq 
\supT \mathbb{E}\left[ 2c_{1}^{2}(1 + 2.3^{2c_{2}+1} \norm{X(s)}^{2c_{2}} ) \right]
=: \tilde{M_{1}}, \nonumber
\end{equation}
and see that $\tilde{M_{1}}<\infty$ by \Cref{thm:moment_bounds}. Together, these definitions yield
\begin{equation}\label{eq_remainded_bound}
\Econd{ \norm{ R_{F}(s,t_{n},X(t_{n}) ) ds } }
\leq 
\sqrt{\Kreg M_{1}} \quad \text{a.s.} 
\end{equation}
Therefore,
\begin{align}
\Econd{ \norm{ \int_{t_{n}}^{t_{n+1}}R_{F}(s,t_{n},X(t_{n})) ds } }
&\leq
\int_{t_{n}}^{t_{n+1}} \Econd{ \norm{ R_{F}(s,t_{n},X(t_{n})) } } ds\\
&\leq
\KF \Delta t_{n+1} \quad a.s. \nonumber 
\end{align}
with $\KF:= \sqrt{\Kreg M_{1}}$. Application of the Cauchy-Schwarz inequality shows that
\begin{align*}
\supTE{  \KF\Delta t_{n+1} }
&\leq
\left( \supTE{ \Kreg } \right)^{1/2} \left( \supTE{ M_{1} } \right)^{1/2} \Delta t_{n+1}  \\
&\leq \sqrt{ \Konep{1} \tilde{M_{1}} } \Delta t_{n+1}^{ \frac{3}{2}-\epsilon}
=: \tilde{\KF} \Delta t_{n+1}^{ \frac{3}{2}-\epsilon}.
\end{align*}

For part $ii)$, we apply Jensen's inequality
\begin{align}
&\Lcondnorm{\int_{t_{n}}^{t_{n+1}}R_{F}(s,t_{n},X(t_{n})) ds}^{2}  \nonumber \\
&\bigbreakgap \leq
\Delta t_{n+1} \int_{t_{n}}^{t_{n+1}}
\Lcondnorm{R_{F}(s,t_{n},X(t_{n}))}^{2}
ds, \nonumber
\end{align}
and consider $\Lcondnorm{R_{F}(s,t_{n},X(t_{n}))}^{2}$. Define
\begin{equation}
D_{q}(s) : = c_{1}(1+2^{c_{2}+1} \norm{X(t_{n})}^{c_{2}} + 2^{c_{2}} \norm{X(s)}^{c_{2}})^{2^{q-1}}\norm{X(s)-X(t_{n})}, \nonumber
\end{equation}
where $c_{1}$ and $c_{2}$ are defined in \Cref{assump_DF_growth}. We note that $D_{q}$ satisfies the following relation
\begin{equation}
D_{q}^{2}(s) = D_{q+1}(s)\norm{X(s)-X(t_{n})}. \nonumber
\end{equation}
Applying Young's inequality to $\Econd{  D_{q}(s) \norm{X(s)-X_{h}(s)} } $ we see that
\begin{align}
\Econd{  D_{q}(s) \norm{X(s)-X_{h}(s)} } 
&\leq \Econdopen{ \norm{X(s)}^{2} + \norm{X(t_{n})}^{2} }\nonumber \\
&\breakgap \Econdclose{ + \frac{c_{1}^{2}}{2}(1+2^{c_{2}+1} \norm{X(t_{n})}^{c_{2}} + 2^{c_{2}} \norm{X(s)}^{c_{2}})^{2^{q}} }\nonumber \\
&=: \bar{D_{q}}, \nonumber
\end{align}
and
\begin{align}
\supTE{ \bar{D_{q}}^{2/(2^{q-1})} }
&\leq
\supT \left( \mathbb{E}[ 4\norm{X(s)}^{4} + c_{1}^{2}(1+ ( 2^{c_{2}+1} + 2^{c_{2}}) \norm{X(s)}^{c_{2}})^{2^{q+1}} ] \right)^{1/(2^{q-1})} \nonumber \\
&=:\tilde{D_{q}} < \infty, \nonumber
\end{align}
due to \Cref{thm:moment_bounds}. Going back to $\Lcondnorm{R_{F}(s,t_{n},X(t_{n}))}^{2}$ we can apply the Cauchy-Schwarz inequality $q$ successive times, for $q > 1-\log_{2}\epsilon$ together with \Cref{lemma_cond_regularity_time} to show that
\begin{align}
\Lcondnorm{R_{F}(s,t_{n},X(t_{n}))}^{2} 
&\leq 
\Econd{D_{1}(s) \norm{X(s)-X(t_{n})}} \nonumber \\
&\leq
\Econd{D_{q}(s) \norm{X(s)-X(t_{n})}}^{1/(2^{q-1})} \nonumber \\
&\breakgap \times \Lcondnorm{X(s)-X(t_{n})}^{2\sum_{i=2}^{q}1/(2^{i-1})} \nonumber \\
&\leq
\bar{D_{q}}^{1/(2^{q-1})}\Kreg^{1-\epsilon} \quad a.s. \nonumber \label{eq:RF_epsilon_bnd}
\end{align}
Therefore 
\begin{align}
&\Lcondnorm{\int_{t_{n}}^{t_{n+1}}R_{F}(s,t_{n},X(t_{n})) ds}^{2}  \nonumber \\
&\bigbreakgap \leq
\Delta t_{n+1} \int_{t_{n}}^{t_{n+1}}
\Lcondnorm{R_{F}(s,t_{n},X(t_{n}))}^{2}
ds \nonumber \\
&\bigbreakgap \leq
\Delta t_{n+1} \int_{t_{n}}^{t_{n+1}} \bar{D_{q}}^{1/(2^{q-1})}\Kreg^{1-\epsilon} ds \quad a.s. \nonumber \\
&\bigbreakgap =:
\Delta t_{n+1}\KFsq. \nonumber
\end{align}
To complete part ii) we apply the Cauchy-Schwarz inequality to show that
\begin{align*}
\supTE{\Delta t_{n+1}K_{3}}
&\leq \Delta t_{n+1}^{2} 
	\left( \supTE{ D_{q}^{2/(2^{q-1})} } \right)^{1/2} 
	\left( \supTE{K_{1}^{2(1-\epsilon)} } \right)^{1/2} \\
&\leq
	\sqrt{ \tilde{D_{q}} \Konep{2(1-\epsilon)} } \Delta t_{n+1}^{3-4\epsilon+2\epsilon^{2} } \\
&<
	\tilde{\KFsq} \Delta t_{n+1}^{3-\epsilon'},
\end{align*}
with $\tilde{\KFsq}:= \sqrt{ \tilde{D_{q}} \Konep{2(1-\epsilon)} }$ and $\epsilon':= 4\epsilon$.

For part $iii)$ first apply the conditional \ito{} isometry, \Cref{eq_ito_isometry_SPDE}, the definition of $R_{B}$ then \Cref{assump_B_lip} to show that
\begin{align}
&\Econd{ \norm{ \int_{t_{n}}^{t_{n+1}}R_{B}(s,t_{n},X(t_{n})) dW(s) }^{2} } \nonumber \\
&\breakgap = 
\Econd{ \int_{t_{n}}^{t_{n+1}} \Bnorm{ B(X(s)) - B(X(t_{n}))  }^{2} ds }
 \nonumber \\
&\breakgap \leq
\LB^{2}\int_{t_{n}}^{t_{n+1}}
\Lcondnorm{X(s)-X(t_{n})}^{2} ds. \nonumber
\end{align}
\Cref{lemma_cond_regularity_time} shows that
\begin{align}
&\Econd{
\norm{ \int_{t_{n}}^{t_{n+1}} R_{B}(s,t_{n},X(t_{n})) dW(s) }^{2}
} \nonumber \\
&\bigbreakgap \leq
\LB^{2}\Kreg \Delta t_{n+1} =: \KB \Delta t_{n+1}.
\end{align}
Finally by \cref{lemma_cond_regularity_time} we have $\supTE{\KB \Delta t_{n+1} }\leq \tilde{\KB} \Delta t_{n+1}^{2-\epsilon}$ with $\tilde{\KB}:= \LB^{2}\Konep{1}$.
\end{proof}

\section{Proofs of Main Results}\label{sec:proofmain}
We first recall several properties of the error in the semigroup approximation $S_{h}(t)$ of $S(t)$ and state a non-linear generalisation of the Gronwall inequality, which we apply to estimate the error in the spatial discretisation.
\begin{lemma}\label{lemma_Th_approx}
Let $0\leq \alpha \leq \beta \leq 2$, $0\leq \nu \leq 1$ and $h\in (0,1]$. For the approximation operator $T_{h}(t):= S(t) - S_{h}(t)P_{h}$, the following error bounds hold
\begin{alignat*}{3}
i) \norm{T_{h}(t)x} 
&\leq \CT h^{\beta}t^{-\frac{(\beta-\alpha)}{2} }\norm{x}_{\alpha}, 
&\text{ for all } x\in \DAs{ \frac{\alpha}{2} }. \\
ii) \norm{\int_{0}^{t} T_{h}(t) x \hspace{0.1cm} ds }
&\leq \CT h^{2-\nu}\norm{x}_{-\nu},
&\text{ for all } x\in \DAs{ \frac{\nu}{2} }.
\end{alignat*}
\end{lemma}
\begin{proof}
For i) see \citep{thomee2006}. Estimate ii) is proven in \citep[lemma 4.2]{Kruse2014}.
\end{proof}
\begin{lemma}\label{lemma:nonlinGronwall}
Let $u(t)$ be a non-negative function which satisfies the inequality
\begin{equation}
u(t) \leq c + \int_{0}^{t} a u(s) + b u^{\delta}(s) ds, \nonumber
\end{equation}
for $a,b,c, \geq 0$. For $0\leq \delta <1$, $u(t)$ can be bounded by
\begin{equation}
u(t) \leq e^{ at}\left( c^{\delta} + \frac{b}{a}(1-e^{-\delta at}) \right)^{1/\delta} \nonumber
\end{equation}
\end{lemma}

\begin{proof}
This is a specific form of the non-linear Gronwall inequality for homogeneous coefficients, which is proved in generality in \citep{willett1965discrete}.
\end{proof}
We decompose the total error at time $t_{n}$ as
\begin{align}
\underbrace{ X(t_{n})-X_{h}^{n} }_{:= E_{n} }
= \underbrace{ X(t_{n})-X_{h}(t_{n}) }_{:= E_{h}(t_{n}) } + 
\underbrace{X_{h}(t_{n})-X_{h}^{n} }_{:= E_{k}(t_{n}) }. \nonumber 
\end{align}
The total error $E_{n}$ can be bounded by
\begin{equation}\label{eq:full_bound}
\Lnorm{E_{n}}^{2} \leq 2 \Lnorm{E_{h}(t_{n})}^{2} + 2\Lnorm{E_{k}(t_{n})}^{2}, 
\end{equation}
and we consider the error terms in space, $E_{h}$, and time, $E_{k}$, separately. 

\begin{proof}[Proof of Theorem \ref{thm:Eh}]

The spatial approximation error $E_{h}$ can be written using the error operator introduced in \Cref{lemma_Th_approx} as,
\begin{align}
E_{h}(T) 
&= 
T_{h}(T)X_{0} +
\int_{0}^{T} T_{h}(T-s)F(X(s))  \nonumber \\
&\breakgap + S_{h}(T-s)P_{h} \Big{(} F(X(s)) - F(X_{h}(s)) \Big{)} ds \nonumber \\
&\breakgap +
\int_{0}^{T} T_{h}(T-s)B(X(s)) + S_{h}(T-s)P_{h} B(X(s)) (I-P_{J}) \nonumber \\
&\breakgap + S_{h}(T-s)P_{h} \Big{(} B(X(s)) - B(X_{h}(s)) \Big{)}P_{J} dW(s). \nonumber
\end{align}
We estimate the error $\Lnorm{E_{h}}^{2}$ using the (standard) \itos{} isometry 
\begin{align}
\Lnorm{E_{h}(T)}^{2}
&\leq
2^{5} \Lnorm{T_{h}(T)X_{0}}^{2} +
2^{5} \Lnorm{ \int_{0}^{T}T_{h}(T-s)F(X(s)) ds }^{2} \nonumber \\
&\breakgap + 
2^{5}\Lnorm{\int_{0}^{T} 
S_{h}(T-s)P_{h} \left( F(X(s)) - F(X_{h}(s)) \right)  ds }^{2} \nonumber \\
&\breakgap +
2^{5}\mathbb{E}\int_{0}^{T} \Bnorm{T_{h}(T-s)B(X(s))}^{2} \nonumber \\
&\bigbreakgap + \Bnorm{S_{h}(T-s)P_{h}(I-P_{J})B(X(s))}^{2} \nonumber \\
&\bigbreakgap +
\Bnorm{S_{h}(T-s)P_{h}P_{J}\Big{(} B(X(s)) - B(X_{h}(s)) \Big{)} }^{2} ds. \nonumber \\
&:= 2^{5}\left( I + II + III + IV + V + VI \right).
\label{eq:Eh}
\end{align}

%
%
We proceed using \Cref{lemma_Th_approx} i) with $\beta=1+r$ and $\alpha=1$ for term I to show that
\begin{equation}\label{eq:spatial:I}
I\leq \CT h^{2(1+r)}T^{-r}\mathbb{E}\norm{X_{0}}^{2}_{1} := \CI{I}h^{2(1+r)}.
\end{equation}

%
%
Part $II$ is estimated using a Taylor expansion of $F(X(s))$ in terms of $F(X(T))$, we have
\begin{align}
II
&:= 
 \Lnorm{ \int_{0}^{T}  T_{h}(T-s)F(X(s)) ds }^{2} \nonumber \\
&=
 \Lnorm{ \int_{0}^{T}  T_{h}(T-s)\left( F(X(T)) + R_{F}(s,T,X(T)) \right) ds}^{2} \nonumber \\
&\leq
 2\Lnorm{\int_{0}^{T}  T_{h}(T-s)F(X(T)) ds }^{2} \nonumber \\
&\breakgap + 
 2\Lnorm{ \int_{0}^{T} T_{h}(T-s)R_{F}(s,T,X(T)) ds}^{2} \nonumber \\
&:= II_{1} + II_{2}. \nonumber
\end{align}

For $II_{1}$ we apply \Cref{lemma_Th_approx} part ii) with $\nu = 1-r$, which requires $r\leq 1$, together with \cref{eq_F_growth} to show,
\begin{align}
II_{1}
&\leq 2\CT h^{2(1+r)} \Lnorm{(-A)^{\frac{-1+r}{2}} F(X(T)) }^{2} \nonumber \\
&\leq 2\CT h^{2(1+r)} \lambda_{1}^{\frac{-1+r}{2}} \Lnorm{F(X(T)) }^{2} \nonumber \\
&\leq 2\CT \lambda_{1}^{\frac{-1+r}{2}} c_{3}^{2} \sup_{s\in [0,T]} \Lnorm{X(s)}^{2(c_{2}+1)}  h^{2(1+r)} \nonumber \\
&:= \CI{II_{1}}h^{2(1+r)}. \label{eq:spatial:II1}
\end{align}
The last line uses \Cref{thm:moment_bounds} for the true solution $X(t)$. Term $II_{2}$ is bounded by \Cref{lemma_Th_approx} i) with $\beta = 1+r-\epsilon$ and $\alpha = 0$, we have
\begin{align}
II_{2} &\leq 
2T  \int_{0}^{T} \Lnorm{T_{h}(T-s)R_{F}(s,T,X(T))}^{2} ds \nonumber \\
&\leq
2\CT T \int_{0}^{T} h^{2(1+r-\epsilon)} (T-s)^{-(1+r-\epsilon)} \Lnorm{R_{F}(s,T,X(T))}^{2} ds. \nonumber
\end{align}
Similarly to the proof of \cref{lemma_remainder_terms} ii), in particular \Cref{eq:RF_epsilon_bnd}, we have that 
\begin{equation}
\Lnorm{R_{F}(s,T,X(T))}^{2} 
\leq 
\bar{D_{q}}^{1/(2^{q-1})}\Kreg^{1-\epsilon} (T-s)^{1-\epsilon } \nonumber,
\end{equation}
for some constant $C_{F}$. Inserting into the bound for $II_{2}$ it follows
\begin{align}
II_{2} 
&\leq
2\CT T \bar{D_{q}}^{1/(2^{q-1})}\Kreg^{1-\epsilon} h^{2(1+r-\epsilon)} \int_{0}^{T}
(T-s)^{-r} ds \nonumber \\
&:=
\CI{II_{2}} h^{2(1+r-\epsilon)} \label{eq:spatial:II2}
\end{align}
In the last step we require $r<1$, otherwise the integral does not converge and $\CI{II_{2}}$ blows up as $r\rightarrow 1$. In practice this blow up is not observed for $r=1$. Equations \Cref{eq:spatial:II1} and \cref{eq:spatial:II2} show that
\begin{equation}\label{eq:spatial:II}
II \leq \CI{II} h^{2(1+r-\epsilon)}, \quad \CI{II}:= \CI{II_{1}} + \CI{II_{2}}.
\end{equation}

%
%
Moving onto $III$, we re-write $F(X(s))$ in terms of $F(X_{h}(s))$ and a remainder term similar to \Cref{lemma_remainder_terms}.

\begin{align}
F(X(s)) 
&= F(X_{h}(s)) \nonumber \\
&\breakgap + \underbrace{\int_{0}^{1} DF \left( X_{h}(s) + \tau (X(s)-X_{h}(s) \right) ( X(s)-X_{h}(s) ) d\tau}_{:=R_{F}(s,s,X_{h}(s))} . \nonumber
\end{align}

Defining the constant
\begin{equation}
\bar{E_{q}}:=
\sup_{s\in [0,T]} \mathbb{E}[ \frac{c_{1}^{2}}{2}(1+2^{c_{2}+1} \norm{X_{h}(s)}^{c_{2}} + 2^{c_{2}} \norm{X(s)}^{c_{2}})^{2^{q}} + \norm{X(s)}^{2} + \norm{X_{h}(s)}^{2} ], \nonumber
\end{equation}
we have that $\bar{E_{q}}<\infty$ by \Cref{thm:moment_bounds} and is independent of $h$ by Proposition 7.3 of \citep{Jentzen2015}. Following the same strategy as the proof of \Cref{lemma_remainder_terms} part ii), we can bound the third term, for any $\delta \in (0,1)$ as

\begin{align}
III&:= 
\Lnorm{\int_{0}^{T} S_{h}(T-s)P_{h}R_{F}(s,s,X_{h}(s)) ds}^{2} \nonumber \\
&\leq 
\CI{III} \int_{0}^{T}  \Lnorm{E_{h}(s)}^{2-\delta} ds, \quad
\CI{III}:=T\bar{E_{q}}^{1/(2^{q-1})}. \label{eq:spatial:III}
\end{align}

%
%
%
Returning to \Cref{eq:Eh}, we proceed to estimate $IV$ by the definition of $\Bnorm{}$ and applying \Cref{lemma_Th_approx} i) with $\beta = 1+r-\epsilon$ and $\alpha = r$ (to the individual terms inside $\Bnorm{}$). That is
\begin{align}
IV
&:=
\int_{0}^{T}\sum_{j=1}^{\infty}
\norm{T_{h}(T-s)B(X(s))Q^{\frac{1}{2} }\chi_{j} }^{2} 
ds \nonumber \\
&\leq
\CT h^{2(1+r-\epsilon)}\int_{0}^{T}\sum_{j=1}^{\infty}
(T-s)^{-1+\epsilon} \norm{A^{\frac{r}{2}} B(X(s))Q^{\frac{1}{2} }\chi_{j} }^{2} 
ds \nonumber \\
&=
\CT h^{2(1+r-\epsilon)}\int_{0}^{T}(T-s)^{-1+\epsilon} \Bnorm{A^{\frac{r}{2}} B(X(s))}^{2} ds. \nonumber
\end{align}
Application of \Cref{assump_B_growth} and \Cref{thm:moment_bounds} completes the estimate of $IV$,
\begin{align}
IV
&\leq
 \frac{\CT T^{\epsilon}}{\epsilon}
h^{2(1+r-\epsilon)}
:= \CI{IV}h^{2(1+r-\epsilon)}.  \label{eq:spatial:IV}
\end{align}

The next term of \Cref{eq:Eh}, $V$, we have for $r\in[0,1)$ and any $\epsilon \in (0,1)$
\begin{align}
V
&:=
\mathbb{E}\int_{0}^{T}\Bnorm{S_{h}(T-s)P_{h}(I-P_{J})B(X(s))}^{2} ds \nonumber \\
&\leq
\mathbb{E}\int_{0}^{T}
\Onorm{(-A)^{\frac{1-\epsilon}{2}} S_{h}(T-s)P_{h}}{H}^{2} \nonumber \\
&\breakgap \times \Onorm{(I-P_{J})(-A)^{-\frac{1+r-\epsilon}{2}}}{H}^{2}
\Bnorm{(-A)^{\frac{r}{2}} B(X(s))}^{2} ds. \nonumber
\end{align}
We use \Cref{lemma_Stv-v} to bound the first term in the integral and \Cref{assump_B_growth} on the third term in the integral to show that
\begin{align}
V
&\leq
\Cs \int_{0}^{T} (T-s)^{-1+\epsilon} ds \hspace{0.1cm} 
\lambda_{J+1}^{-(1+r)+\epsilon}
(1+\sup_{s\in [0,t]} \mathbb{E}\norm{X(s)}^{2}_{r})
 \nonumber \\
&=
\frac{ \Cs T^{\epsilon} }{\epsilon} (1+\sup_{s\in [0,t]} \mathbb{E}\norm{X(s)}^{2}_{r})  \lambda_{J+1}^{-(1+r)+\epsilon} \nonumber \\
&:= \CI{V}\lambda_{J+1}^{-(1+r)+\epsilon}. \label{eq:spatial:V}
\end{align}

The final term $VI$ of \cref{eq:Eh} is estimated by \Cref{assump_B_lip} together with the boundedness of the semigroup and the boundedness of $\Onorm{P_{J}}{H}$ and $\Onorm{P_{h}}{H}$. We have \begin{align}
VI &:= \mathbb{E}\int_{0}^{T}\Bnorm{S_{h}(T-s)P_{h}P_{J}\Big{(} B(X(s)) - B(X_{h}(s)) \Big{)} }^{2} ds \nonumber \\
&\leq
\LB^{2}\int_{0}^{T} \ \Lnorm{E_{h}(s)}^{2} ds. \label{eq:spatial:VI}
\end{align}

Returning to \Cref{eq:Eh}, \cref{eq:spatial:I,eq:spatial:II,eq:spatial:III,eq:spatial:IV,eq:spatial:V,eq:spatial:VI}
show that (for $C_{1}:= \CI{I} + \CI{II} + \CI{IV}$)
\begin{align}
\Lnorm{E_{h}(T)}^{2}
&\leq
C_{1}h^{2(1+r-\epsilon)} + \CI{V}\lambda_{J+1}^{-(1+r)+\epsilon} \nonumber \\
&\breakgap +  \LB^{2}\int_{0}^{T}  \Lnorm{E_{h}(s)}^{2} ds \nonumber \\
&\breakgap + \CI{III}\int_{0}^{T} \Lnorm{E_{h}(s)}^{2-\delta} ds.\nonumber
\end{align}

Application of \Cref{lemma:nonlinGronwall} yields
\begin{align}
\Lnorm{E_{h}(T)}^{2}
&\leq
e^{\LB^{2}T}\left[
\left(
C_{1}h^{2(1+r-\epsilon)}+ 
\CI{IV}\lambda_{J+1}^{-(1+r)+\epsilon}
\right)^{\delta} \right. \nonumber \\
&\breakgap + \left. \frac{\CI{III}}{\LB^{2}}(1-e^{-\epsilon \LB^{2}T})
\right]^{1/\delta} \nonumber \\
&\leq
e^{\LB^{2}T}\left[
\left(
C_{1}h^{2(1+r-\epsilon)}+ 
\CI{IV}\lambda_{J+1}^{-(1+r)+\epsilon}
\right)^{\delta}
+ \CI{III}T \delta
\right]^{1/\delta} \nonumber \\
&\leq
e^{\CI{III}T}e^{e^{1/e}} e^{\LB^{2}T} 
\left(
C_{1}h^{2(1+r-\epsilon)}+ 
\CI{IV}\lambda_{J+1}^{-(1+r)+\epsilon}
\right) \nonumber .
\end{align}

\end{proof}

\begin{proof}[Proof of Theorem \ref{thm:Ek}]
The time discretisation error, $E_{k}(t_{n})$, must be considered over one step with conditional expectation. This is to ensure the time-step size selection, computed using the current solution, does not bias the numerical scheme. The time discretisation error can be written as
\begin{align}
& E_{k}(t_{n+1}) = 
S_{h}(t_{n+1}-t_{n})P_{h}E_{k}(t_{n}) \nonumber \\
& \breakgap + 
\int_{t_{n}}^{t_{n+1}} S_{h}(t_{n+1}-s)P_{h} F(X_{h}(s)) - S_{h}(t_{n+1}-s)P_{h}F(X_{h}^{n})  ds \nonumber\\
& \breakgap +
\int_{t_{n}}^{t_{n+1}} S_{h}(t_{n+1}-s)P_{h}B(X_{h}(s)) - 
S_{h}(t_{n+1}-t_{n})P_{h}B(X_{h}^{n})  dW(s) \nonumber
\end{align}
Over one step, we Taylor expand the drift and diffusion coefficients around $X_{h}(t_{n})$ as follows,

\begin{align}
E_{k}(t_{n+1})
&=
S_{h}(t_{n+1}-t_{n})P_{h}E_{k}(t_{n}) \nonumber \\
& \breakgap + 
\underbrace{
\int_{t_{n}}^{t_{n+1}} S_{h}(t_{n+1}-s)P_{h}\Big{(} F(X_{h}(t_{n})) - F(X_{h}^{n}) \Big{)} ds
}_{:= I} \nonumber \\
&\breakgap +
\underbrace{
\int_{t_{n}}^{t_{n+1}} 
(S_{h}(t_{n+1}-s) - S_{h}(t_{n+1}-t_{n}))P_{h}B(X_{h}(t_{n})) dW(s)}_{:= II_{1}} \nonumber \\
&\breakgap +
\underbrace{S_{h}(t_{n+1}-t_{n})P_{h}(B(X_{h}(t_{n}))-B(X_{h}^{n}))\Delta W_{n+1}
}_{:= II_{2}} \nonumber \\
&\breakgap +
\underbrace{
\int_{t_{n}}^{t_{n+1}} S_{h}(t_{n+1}-s)P_{h} R_{F}(s,t_{n},X_{h}(t_{n})) ds
}_{:= \tilde{R}_{F} } \nonumber\\
&\breakgap +
\underbrace{
\int_{t_{n}}^{t_{n+1}} S_{h}(t_{n+1}-s)P_{h} R_{B}(s,t_{n},X_{h}(t_{n})) dW(s)
}_{:= \tilde{R}_{B} } \nonumber\\
&:=
S_{h}(t_{n+1}-t_{n})P_{h}E_{k}(t_{n}) + R_{K}, \label{eq_Ek_def}
\end{align}
where we have defined $R_{K} := I + II + \tilde{R}_{F} + \tilde{R}_{B}$ and $II:= II_{1} + II_{2}$. Applying the operator $S_{h}(T-t_{n+1})$ to \Cref{eq_Ek_def} yields
\begin{align}\label{eq:St_Ek}
\SE{n+1}
&=
\SM{n}P_{h}E_{k}(t_{n}) + \SM{n+1}R_{K}(t_{n+1}). \nonumber \\
&:=
E_{S}^{n} + I_{S} + II_{S} + \tilde{R}_{SF} + \tilde{R}_{SB}. 
\end{align}
The norm $\norm{E_{S}^{n+1}}^{2}$ is expanded and re-written using the inequality 
\\
$\langle A,B \rangle \leq \frac{1}{2}( \norm{A}^{2} + \norm{B}^{2} )$ as
\begin{align}
\norm{E_{S}^{n+1}}^{2}
&= 
\langle E_{S}^{n} + I_{S} + II_{S} + \tilde{R}_{SF} + \tilde{R}_{SB} , E_{S}^{n+1}\rangle \nonumber\\
&\leq
\frac{1}{2} ( \norm{E_{S}^{n}}^{2} + \norm{E_{S}^{n+1}}^{2} ) 
+
\langle I_{S} + II_{S} + \tilde{R}_{SF} + \tilde{R}_{SB},  E_{S}^{n+1}\rangle. \nonumber
\end{align}
Therefore the error over one time-step can be bound by
\begin{equation}
\norm{E_{S}^{n+1}}^{2} - \norm{E_{S}^{n}}^{2}
\leq 2 \langle  I_{S} + II_{S} + \tilde{R}_{SF} + \tilde{R}_{SB}, E_{S}^{n+1} \rangle. \nonumber
\end{equation}
Expansion of both $R_{SK}^{n+1}$ and $E_{S}^{n+1}$ inside the inner product then collation of the stochastic terms shows that
\begin{align}
\norm{E_{S}^{n+1}}^{2} - \norm{E_{S}^{n}}^{2}
&\leq
2\Big{(} 
\langle I_{S}, E_{S}^{n} + I_{S} + II_{S} + \tilde{R}_{SF} + \tilde{R}_{SB} \rangle \nonumber \\
&+
\langle \tilde{R}_{SF}, E_{S}^{n} + I_{S} + II_{S} + \tilde{R}_{SF} + \tilde{R}_{SB} \rangle \nonumber \\
&+
\langle II_{S} + \tilde{R}_{SB}, E_{S}^{n} + I_{S} + II_{S} + \tilde{R}_{SF} + \tilde{R}_{SB} \rangle 
\Big{)} \nonumber .
\end{align}

Further applications of the inequality $\langle A,B \rangle \leq \frac{1}{2}( \norm{A}^{2} + \norm{B}^{2} )$ and rearrangement yields,
\begin{align}
\norm{E_{S}^{n+1}}^{2} - \norm{E_{S}^{n}}^{2}
&\leq
2\langle I_{S}, E_{S}^{n} \rangle + 2\langle \tilde{R}_{SF}, E_{S}^{n} \rangle \nonumber \\
&+
4\norm{I_{S}}^{2} +
7\norm{\tilde{R}_{SF}}^{2} +
8\norm{II_{S}}^{2} +
8\norm{\tilde{R}_{SB}}^{2} \nonumber \\
&+ 
2\langle 2I_{S}+ E_{S}^{n}, II_{S} + \tilde{R}_{SB} \rangle. \label{eq_Ek_inner_prod_master}
\end{align}

The expected value, conditioned on the discrete filtration $\mathcal{G}_{n}$, of the final term $\langle 2I_{S}+ E_{S}^{n}, II_{S} + \tilde{R}_{SB} \rangle$, is zero. This is true since both $I_{S}$ and $E_{S}^{n}$ are $\mathcal{G}_{n}$ measurable and $\mathbb{E}[ II_{S} + \tilde{R}_{SB} \cond{G}{n}=0$. We bound each of the terms in \Cref{eq_Ek_inner_prod_master} in turn. Boundedness of the semigroup and \Cref{assump_F_local_lipschitz} on the drift are applied to the first term to show
\begin{align}
&\langle I_{S}, E_{S}^{n} \rangle \nonumber \\
&= 
\left\langle \int_{t_{n}}^{t_{n+1}} S_{h}(T-s)P_{h}\Big{(} F(X_{h}(t_{n})) - F(X_{h}^{n}) \Big{)} ds, S_{h}(T-t_{n})P_{h}E_{k}(t_{n}) \right\rangle  \nonumber\\
&\leq
\Delta t_{n+1}
\left\langle F(X_{h}(t_{n})) - F(X_{h}^{n}), E_{k}(t_{n}) \right\rangle \nonumber\\
&\leq
\LF\Delta t_{n+1}\norm{E_{k}(t_{n})}^{2}. \label{eq_EK_prod_IS}
\end{align}

The second term is estimated using the fact $\norm{E_{S}^{n}}$ is $\mathcal{G}_{n}$ measurable, \Cref{lemma_remainder_terms}, the boundedness of the semigroup and Young's inequality, 
\begin{align}
\Econd{ \langle \tilde{R}_{SF}, E_{S}^{n} \rangle }
&\leq
\Econd{ \norm{\tilde{R}_{F}}\norm{E_{S}^{n}} } \nonumber\\
&\leq
\norm{E_{k}(t_{n})}\Econd{ \norm{\tilde{R}_{F}}} \nonumber\\
&\leq \KF \Delta t_{n+1}\norm{E_{k}(t_{n})} \quad a.s.\nonumber\\
&\leq \frac{1}{2}\KF^{2} \Delta t_{n+1} + \frac{1}{2} \Delta t_{n+1}\norm{E_{k}(t_{n})}^{2} \quad a.s. \label{eq_EK_prod_RSF}
\end{align}

For $\norm{I_{S}}^{2}$ we use Jensen's inequality, the boundedness of the semigroup and the admissibility of the time-stepping strategy,
\begin{align}
\norm{I_{S}}^{2}
&=
\norm{\int_{t_{n}}^{t_{n+1}} 
S_{h}(T-s)P_{h}\Big{(} F(X_{h}(t_{n})) - F(X_{h}^{n}) \Big{)} ds}^{2} \nonumber\\
&\leq
\Big{(}\int_{t_{n}}^{t_{n+1}} 
\norm{S_{h}(T-s)P_{h}}_{\mathcal{L}(H)} \norm{F(X_{h}(t_{n})) - F(X_{h}^{n})} ds \Big{)}^{2} \nonumber\\
&\leq
\Delta t_{n+1}^{2}\norm{F(X_{h}(t_{n})) - F(X_{h}^{n})}^{2} \nonumber\\
&\leq
2\Delta t_{n+1}^{2} \left( \norm{F(X_{h}(t_{n}))}^{2} + \norm{F(X_{h}^{n})}^{2} \right) \nonumber\\
&\leq
2\Delta t_{n+1}^{2} \left( \norm{F(X_{h}(t_{n}))}^{2} + 
R_{1} + R_{2}\norm{X_{h}^{n}}^{2} \right) \nonumber\\
&\leq
2\Delta t_{n+1}^{2} 
\left( (1+\norm{X_{h}(t_{n})}^{c+1})^{2} + R_{1}  \nonumber \right. \\
& \quad\quad + 2R_{2}\norm{X_{h}(t_{n})}^{2} + \left.2R_{2}\norm{E_{k}(t_{n})}^{2}
\right). \nonumber
\end{align}
Therefore
\begin{equation}
\Lcondnorm{I_{S}}^{2} \leq K_{5}\Delta t_{n+1}^{2} + 4R_{2}\Delta t_{n+1}^{2}\norm{E_{k}(t_{n})}^{2} \quad a.s, \label{eq_EK_IS}
\end{equation}
for the a.s. finite $\mathcal{G}_{n}$-measurable random variable (a.s. finite and independent of $h$ by Proposition 7.3 of \citep{Jentzen2015})
\begin{align}
K_{5}&:= 2\left( 2+ 2\Econd{ \norm{X_{h}(t_{n})}^{2(c+1)} } + R_{1} + 2R_{2}\Econd{ \norm{X_{h}(t_{n})}^{2} } \right). \nonumber
\end{align}
The random variable $K_{5}$ will later be bounded by
\begin{align*}
\supTE{K_{5}} 
&\leq 4+R_{1} + 4\supTE{ \norm{X_{h}(s)}^{2(c+1)} } + 2R_{2} \supTE{ \norm{X_{h}(s)}^{2} } \nonumber \\
&=: \tilde{K_{5}}.
\end{align*}
\Cref{lemma_remainder_terms} and the boundedness of the semigroup implies
\begin{align}
\Lcondnorm{ \tilde{R}_{SF} }^{2}
&\leq 
\KFsq \Delta t_{n+1} \quad a.s, \label{eq:EK:RSF} \\
\Lcondnorm{\tilde{R}_{SB}}^{2} 
&\leq
\KB \Delta t_{n+1} \quad a.s. \label{eq:EK:RSB} 
\end{align}
The next term to estimate in \Cref{eq_Ek_inner_prod_master} is $\norm{II_{S}}^{2}$, 
\begin{align}
&\Lcondnorm{II_{S}}^{2} \nonumber \\
&\leq
2\Econd{ \norm{\int_{t_{n}}^{t_{n+1}} 
(S_{h}(T-s) - S_{h}(T-t_{n}))P_{h}B(X_{h}(t_{n})) 
dW(s)}^{2} } \nonumber\\
&\breakgap +
2\Econd{ \norm{S_{h}(T-t_{n})P_{h}(B(X_{h}(t_{n}))-B(X_{h}^{n})) \Delta W_{n+1}}^{2} } \nonumber\\
&:= II_{S1} + II_{S2}. \nonumber
\end{align}

Firstly using \Cref{thm_cond_ito_isometry}, \Cref{assump_B_lip} and \Cref{lemma_Stv-v} we have for $\epsilon \in (0,1)$ that $II_{S1}$ can be bounded by
\begin{align}
II_{S1}
&\leq
2\Econd{ \int_{t_{n}}^{t_{n+1}} 
\Onorm{ (-A_{h})^{\frac{1-\epsilon}{2}}S_{h}(T-s) }{H}^{2} \right. \nonumber \\
&\breakgap\times \left.\Onorm{(-A_{h})^{-\frac{1-\epsilon}{2}}(I - S_{h}(s-t_{n}) )P_{h}}{H}^{2}
\Bnorm{B(X_{h}(t_{n})) }^{2} ds } \nonumber\\
&\leq
2\Cs^{4} \Delta t_{n+1}^{1-\epsilon}\int_{t_{n}}^{t_{n+1}} (T-s)^{-1+\epsilon} ds
\hspace{0.1cm}
\Econd{
\norm{B(X_{h}(t_{n})) 
}^{2}_{L^{2}_{0}} 
} \nonumber \\
&\leq
2K_{6}\Delta t_{n+1}^{1-\epsilon}\int_{t_{n}}^{t_{n+1}} (T-s)^{-1+\epsilon} ds \quad a.s. \label{eq_EK_II_S1}
\end{align}
With the a.s. finite $\mathcal{G}_{n}$-measurable random variable $K_{6}$ defined as
\begin{equation}
K_{6}:= 2\Cs^{4}\LBr \Econd{2+2\norm{X_{h}(t_{n})}^{2} }, \nonumber
\end{equation}
and later bound by
\begin{equation}
\supTE{K_{6}} \leq 2\Cs^{4}\LBr \supTE{2+2\norm{X_{h}(s)}^{2} }=:\tilde{K_{6}}.
\end{equation}
To bound $II_{S2}$ we use \Cref{thm_cond_ito_isometry}, the boundedness of the semigroup and $P_{h}$ operator with assumption \ref{assump_B_lip} to show
\begin{align}
II_{S2} 
&= 
2\Econd{ \norm{S_{h}(T-t_{n})P_{h}(B(X_{h}(t_{n}))-B(X_{h}^{n})) \Delta W_{n+1}}^{2}
}\nonumber \\
&\leq
2\Delta t_{n+1} \norm{B(X_{h}(t_{n}))-B(X_{h}^{n})}^{2}_{L^{2}_{0}} \nonumber\\
&\leq
2\LB^{2}\Delta t_{n+1} \norm{E_{k}(t_{n})}^{2}. \label{eq_EK_II_S2}
\end{align}

Combining the estimates in  \Cref{eq_EK_prod_IS,eq_EK_prod_RSF,eq:EK:RSF,eq:EK:RSB,eq_EK_IS,eq_EK_II_S1,eq_EK_II_S2} for \Cref{eq_Ek_inner_prod_master} implies
\begin{align}
\Econd{ \norm{E_{S}^{n+1}}^{2} } - 
\norm{E_{S}^{n}}^{2} 
&\leq 
C_{1}\Delta t_{n+1}
+
C_{2}\Delta t_{n+1} \norm{E_{k}(t_{n})}^{2}  \nonumber \\
&\breakgap +
C_{3}\Delta t_{n+1}^{1-\epsilon}\int_{t_{n}}^{t_{n+1}}(T-s)^{-1+\epsilon} ds, \quad a.s. \nonumber 
\end{align}
with $C_{1}:= \KF^{2} + 7\KFsq + 8\KB + 4K_{5}\Delta t_{n+1}$, 
$C_{2}:=2\LF + 16R_{2}\Delta t_{n} + 16\LB^{2}$ and
$C_{3}:= 16K_{6}$. Summing both sides from $n=0$ to $n=N-1$ and taking expectations, noting that $E_{S}^{N}:=S(T-T)E_{k}(T)=E_{k}(T)$, shows
\begin{align}
\mathbb{E}[ \norm{E_{k}(T)}^{2} ]
&\leq
\sum_{n=0}^{N-1}\Delta t_{n+1} \mathbb{E}[C_{1}] + 
\sum_{n=0}^{N-1}\Delta t_{n+1}C_{2} \mathbb{E}[\norm{E_{k}(t_{n})}^{2}] \nonumber \\
&\breakgap + 
\Delta t_{\max}^{1-\epsilon} \mathbb{E}[C_{3}]
\int_{0}^{T}(T-s)^{-1+\epsilon} ds.  \nonumber 
\end{align}
Taking supremum over $s\in [0,T]$ on the right hand side yields
\begin{align}
\mathbb{E}[ \norm{E_{k}(T)}^{2} ]
&\leq
\frac{T^{\epsilon} }{\epsilon}\Delta t_{\max}^{1-\epsilon}\supTE{C_{3}}  +
\sum_{n=0}^{N-1}\Delta t_{n+1}\supTE{ C_{1} } + 
\sum_{n=0}^{N-1}\Delta t_{n+1}C_{2}\norm{E_{k}(t_{n})}^{2}.
\label{eq_Ek_sum}
\end{align}
We apply \cref{lemma_cond_regularity_time,lemma_remainder_terms} to see that
\begin{align}
\supTE{ C_{1} } \leq \tilde{M_{1}}^{2}\Konep{2} \Delta t_{n+1}^{1-\epsilon} + 7\tilde{\KFsq}\Delta t_{n+1}^{2-\epsilon} + 8\tilde{\KB} \Delta t_{n+1}^{1-\epsilon} + 4\tilde{K_{5}}\Delta t_{n+1} =: \Delta t_{n+1}^{1-\epsilon}\Gamma, \nonumber
\end{align}
and
\begin{equation}
\supTE{C_{3}} \leq 16 \tilde{K_{6}} \nonumber .
\end{equation}
Defining $\Lambda:= \Gamma + \frac{16 \tilde{K_{6}}T^{\epsilon} }{\epsilon}$ and substitution into \cref{eq_Ek_sum} shows that
\begin{equation}
\mathbb{E}[ \norm{E_{k}(T)}^{2} ]
\leq
\Lambda \Delta t_{\max}^{1-\epsilon} + 
\sum_{n=0}^{N-1}\Delta t_{n+1}C_{2}\norm{E_{k}(t_{n})}^{2}. \label{eq_Ek_sum2}
\end{equation}

Finally, application of the discrete Gronwall lemma to \Cref{eq_Ek_sum2} yields,
\begin{align}
\mathbb{E}[ \norm{E_{k}(T)}^{2} \cond{G}{n}
&\leq 
\Lambda \Delta t_{\max}^{1-\epsilon}
\exp (C_{2}T). \label{eq_Ek_bound}
\end{align}
Combining the estimate from \Cref{thm:Ek} with \Cref{eq_Ek_bound} yields the result.
\end{proof}

To conclude the section we give the proof of \Cref{corr:phi0}. 

\begin{proof}[Proof of Corollary \ref{corr:phi0}]
If the $\varphi_{0}$ approximation is used instead of $\varphi_{1}$ for the drift contribution, that is approximations \cref{eq:phi0_approx} instead of \cref{eq:phi1_approx}, an additional term
\begin{equation}
III_{S}
:=
\int_{t_{n}}^{t_{n+1}} \left( S_{h}(T-s) - S_{h}(T-t_{n}) \right) P_{h} F(X_{h}(t_{n})) ds \nonumber ,
\end{equation}
appears in \cref{eq:St_Ek}. This introduces additional inner product terms against the existing terms in \cref{eq_Ek_inner_prod_master}. All terms not involving $III_{S}$ can be dealt with as before. The term $III_{S}$ is $\mathcal{G}_{n}$-measurable and therefore the inner product with the stochastic terms has zero mean. The norm of $III_{S}$ can be bound as,
\begin{align}
&\Econd{\norm{III_{S}}^{2}} \nonumber \\
&:=
\Econd{\norm{
\int_{t_{n}}^{t_{n+1}} \left( S_{h}(T-s) - S_{h}(T-t_{n}) \right) P_{h} F(X_{h}(t_{n})) ds}^{2}} \nonumber \\
&\leq
\Delta t_{n+1}\int_{t_{n}}^{t_{n+1}} 
\Onorm{ (-A_{h})^{\frac{1-\epsilon}{2}}S_{h}(T-s) }{H}^{2} \nonumber \\
&\breakgap\times
\Onorm{(-A_{h})^{-\frac{1-\epsilon}{2}}(I - S_{h}(s-t_{n}) )}{H}^{2} ds \hspace{0.2cm}
\Econd{\norm{P_{h}F(X_{h}(t_{n})) }^{2}}  \nonumber \\
&\leq
\Cs^{4}\Delta t_{n+1}^{2-\epsilon} \int_{t_{n}}^{t_{n+1}}(T-s)^{-1+\epsilon} ds. \nonumber
\end{align}
The term $\Econd{ \langle III_{S}, E_{S}^{n} \rangle}$ is bounded via the boundedness of the semigroup, Jensen's inequality and Young's inequality
\begin{align}
\Econd{ \langle III_{S}, E_{S}^{n} \rangle}
&\leq 
\norm{E_{S}^{n}}\Econd{ \norm{III_{S}} } \nonumber \\
&\leq \norm{E_{S}^{n}}\left(\Econd{ \norm{III_{S} }^{2} } \right)^{1/2} \nonumber \\
&\leq
\Cs^{2}\Delta t_{n+1}^{1-\epsilon/2}\norm{E_{k}(t_{n})}\left( \int_{t_{n}}^{t_{n+1}}(T-s)^{-1+\epsilon} ds\right)^{\frac{1}{2}}. \nonumber \\
&\leq
\Cs^{4}\Delta t_{n+1}^{1-\epsilon}\int_{t_{n}}^{t_{n+1}}(T-s)^{-1+\epsilon} ds + \Delta t_{n+1}\norm{E_{k}(t_{n})}^{2}. \nonumber
\end{align}
The two remaining inner products of $III$ against $I_{S}$ and $\tilde{R}_{SF}$ contain higher powers of $\Delta t_{n+1}$, leaving the final order of convergence unchanged after application of the discrete Gronwall lemma.
\end{proof}

\begin{remark}
We remark upon the possible dependence of the constants $\Lambda$ and $C_{2}$ in \cref{eq_Ek_bound} on $\rho$. The constant $\rho$ may enter the analysis via the constants $R_{1}$ and $R_{2}$ of \Cref{defn:addmissable}. Upon inspection of the proof of \Cref{thm:Ek} it is seen that $R_{1}$ and $R_{2}$ will be present linearly in $\Lambda$ and $\Delta t_{n+1}R_{2}$ will be present in $C_{2}$. The choice of time-stepping rule then determines the dependence of the overall error constant $\Cmain$ on $\rho$. For example, if rule i) of \Cref{lemma_timestep_selection} is utilized, then $R_{1}=\rho^{2\theta}$ and $R_{2}=0$, hence the constant $\Cmain \propto \rho^{2\theta}$.
\end{remark}
\section{Conclusion}
We have constructed an adaptive time-stepping scheme for semilinear SPDEs with non-Lipschitz drift coefficients. We have proven the scheme strongly convergences, with respect to $\Delta t_{\max}$, with the same rate as for globally Lipschitz drift. It is important to note the convergence of the adaptive scheme is valid only when the time-step size stays above $\Delta t_{\min}$. If the minimum time-step size is ever required then an alternative ``backstop'' scheme must be employed to ensure overall convergence. In practice, we have shown that for moderate choices in the fixed ratio $\Delta t_{\max}/\Delta t_{\min}$, we always maintain a time-step size above $\Delta t_{\min}$. Two numerical experiments in \Cref{sec:numerics} confirm the theoretical convergence rates and highlight the efficiency gains over alternative fixed time-step methods for general non-Lipschitz drift.
\section*{Acknowledgements}
The authors would like to thank C\'onall Kelly (University College Cork) for his helpful discussions.

\bibliography{/home/sc58/Documents/latex/references/library}{}

\end{document}